\documentclass[12pt]{article}
\usepackage{amsmath,amssymb,a4wide,amsthm,mathtools,bm}
\usepackage{hyperref}
\usepackage[shortlabels]{enumitem}
\usepackage{xcolor}
\usepackage{comment}
\usepackage{graphicx}
\usepackage{csquotes}
\usepackage{relsize}
\usepackage[utf8]{inputenc}
\usepackage{todonotes}
\usepackage{authblk}

\usepackage[firstinits=true,backend=bibtex]{biblatex}
\usepackage{bm}

\newtheorem{theorem}{Theorem}
\newtheorem{proposition}[theorem]{Proposition}
\newtheorem{corollary}[theorem]{Corollary}
\newtheorem{lemma}[theorem]{Lemma}
\newtheorem{algorithm}{Algorithm}

\theoremstyle{remark}
\newtheorem{remark}[theorem]{Remark}

\newtheoremstyle{assumptionstyle}
{3pt}
{3pt}
{\itshape}
{}
{\itshape}
{:}
{.5em}
{}

\theoremstyle{assumptionstyle}
\newtheorem*{assumption}{Assumption}
		
\allowdisplaybreaks

\usepackage{float}
\floatstyle{ruled}
\newfloat{algorithm}{h}{loa}
\floatname{algorithm}{Algorithm}
\usepackage{algpseudocode}
\algnewcommand\AND{\textbf{ and }}

\def\({\left(}
\def\){\right)}
\def\l[{\left[}
\def\r]{\right]}

\newcommand{\embed}{\hookrightarrow}

\NewDocumentCommand{\hl}{+m}{#1%
}

\newcommand{\re}[1]{\Re\left(#1\right)}
\newcommand{\im}[1]{\Im\left({#1}\right)}

\newcommand{\diffq}[2]{\frac{\partial #1}{\partial #2}}
\newcommand{\base}[1]{\mathbf{e}_{#1}}
\newcommand{\Rco}{\mathbf{R}}
\newcommand{\Mmat}{\mathbf{M}}
\newcommand{\Lapmh}{\Delta_m}  
\newcommand{\Lapm}{\Delta_m}
\newcommand{\Laph}{\Delta_{\mathrm{h}}}
\newcommand{\vLaph}{\bm{\Delta}_{\mathrm{h}}}

\def\Z{\mathbb{Z}}
\def\R{\mathbb{R}}
\def\C{\mathbb{C}}
\def\D{\,\mathrm{d}}

\def\fh{\hat{f}}

\def\Psih{\widehat{\Psi}}

\def\fv{\textbf{\textit{f}}}
\def\fvh{\fv_{\!\mathrm{h}}}

\def\uv{\textbf{\textit{u}}}
\def\uvh{\uv_{\mathrm{h}}}
\def\rv{\textbf{\textit{r}}}
\def\Mc{{r{\mathbb{S}^2}}}	

\def\vv{\textbf{\textit{v}}}

\def\uvn{{\uv_0}}

\def\Nc{\mathcal{N}}

\def\Xtil{\widetilde{X}}

\def\Omv{\mathbf{\Omega}}
\def\Orefv{\mathbf{\Omega}_{\mathrm{ref}}}
\def\Oref{\Omega_{\mathrm{{ref}}}}
\def\Odag{\Omega^\dagger}

\def\N{\mathbb{N}}

\def\R{\mathbb{R}}
\def\N{\mathbb{N}}

\def\Dc{\mathcal{D}}

\def\ball{B_R^{\R\times X}(\gamma^\dagger,\beta^\dagger)}
\def\balltil{B_R^{\R\times X}(\gamma^\dagger,\Omega^\dagger)}

\def\Bb{\mathbf{B}_m}
\newcommand{\Bbtwo}[2]{\Bb(#1)#2}  
\newcommand{\Bbone}[1]{\Bb(#1)}
\newcommand{\Bbp}[1]{\Bb'(#1)}
\newcommand{\Bbptwo}[2]{\Bb'(#1)#2}
\def\Bf{\mathbf{B}}

\newcommand{\Hd}[1]{H^{#1}_{\diamond}(\Mc)}
\newcommand{\Hdm}[1]{H^{#1}_{\diamond}(\Mc;m)}

\def\cc{\,,\,}

\def\dt{\partial_\theta}
\def\ran{\right\rangle}\def\lan{\left\langle}

\def\lsp{\left\langle}
\def\rsp{\right\rangle}
\newcommand{\shorten}[1]{}
\newcommand{\covder}[1]{\nabla_{\!#1}}

\DeclareMathOperator{\diver}{div}  
\DeclareMathOperator{\vdiver}{\mathbf{div}}  
\DeclareMathOperator{\grad}{\mathbf{grad}}  
\DeclareMathOperator{\curl}{curl}  
\DeclareMathOperator{\Div}{div_{\mathrm{h}}}    
\DeclareMathOperator{\Grad}{\mathbf{grad}_{\mathrm{h}}}  
\DeclareMathOperator{\Gradm}{\mathbf{grad}_{\mathit{m}\!}}
\DeclareMathOperator{\Curl}{curl_{\mathrm{h}}}    
\DeclareMathOperator{\vCurl}{\mathbf{curl}_{\mathrm{h}}}  

\title{\Large\bf
Linear toroidal-inertial waves on a differentially rotating sphere with application to helioseismology:\\ Modeling, forward and inverse problems}
\author[1]{\normalsize Tram Thi Ngoc Nguyen}
\author[1]{Damien Fournier}
\author[1,2]{Laurent Gizon}
\author[1,3]{Thorsten Hohage}
\affil[1]{\small Max Planck Institute for Solar System Research, G\"ottingen, Germany}
\affil[2]{Institute for Astrophysics and Geophysics, University of G\"ottingen, Germany}
\affil[3]{Institute for Numerical and Applied Mathematics, University of G\"ottingen, Germany}

\date{ }
\begin{document}

\pagenumbering{arabic}

\maketitle

\abstract{
This paper develops a mathematical framework for interpreting observations of solar inertial waves in an idealized setting. Under the assumption of purely toroidal linear waves on the sphere, the stream function of the flow satisfies a fourth-order scalar equation. We prove well-posedness of wave solutions under explicit conditions on differential rotation. Moreover, we study the inverse problem of simultaneously reconstructing viscosity and differential rotation parameters from either complete or partial surface data. We establish convergence guarantee of iterative regularization methods by verifying the tangential cone condition, and prove local unique identifiability of the unknown parameters. Numerical experiments with Nesterov-Landweber iteration confirm reconstruction robustness across different observation strategies and noise levels.}
\paragraph{Keywords.} Inverse problems, differential equations on manifolds, 
fourth-order elliptic differential equations, 
helioseismology, inertial waves, iterative regularization, tangential cone condition.

\paragraph{Acknowledgment.} Support from Deutsche Forschungsgemeinschaft (DFG, German Research Foundation) through SFB 1456/432680300 Mathematics of Experiment, Project C04 is gratefully acknowledged. LG and DF acknowledge funding from the ERC Synergy Grant WHOLESUN 810218. \hl{The authors greatly appreciate the feedback received from the reviewers, which has led to an improvement of the manuscript.}

\newcommand{\LG}[1]{\textcolor{red}{[#1]}}

\section{Introduction}

\emph{Helioseismology} is the study of the internal structure and dynamics of the Sun from observations of solar oscillations on the surface. Such inference is primarily done by analyzing the observed pressure waves (p-modes) with periods of order five minutes. Acoustic-mode helioseismology based on the interpretation of the mode frequencies has led to many achievements, including the determination of the Sun's differential rotation as a function of radius and unsigned latitude \cite{Thompson03}. Inversions of the two-point correlations of the wave field at the surface can also be used to infer 3D perturbations in sound speed and flows  in the solar interior \cite{GizonBirchReview, Bjoern23}.

The recent discovery of global Rossby modes on the Sun \cite{Loptien2018} has opened a new research topic: inertial-mode helioseismology. \emph{Solar inertial oscillations} are modes restored by the Coriolis force, with periods spanning several weeks, comparable to solar rotation period of $\sim27$ days at the equator.
These modes require very long observations from space-based and ground-based observatories to be studied with sufficient frequency resolution; 50 years of data are now available \cite{Gizon2024, Liang2025}. Since many of these modes have maximum kinetic energy deep in the convective envelope of the Sun \cite{GizonInertial21}, their study promises to reveal new insights into the physics and  dynamics of the solar interior.  Current efforts focus on the development of simplified physical models of these inertial modes to investigate their unique sensitivity to internal properties of the Sun, such as differential rotation at very high latitudes and the internal turbulent viscosity. 
Several highly idealized 2D models focusing on purely toroidal modes have been developed for both $\beta$-plane \cite{Gizon2020} and spherical geometries \cite{Damien22}.
The dynamics of these retrograde-propagating modes is governed by the functional form of the background latitudinal differential rotation, with slower rotation at high latitudes than at the equator. Particularly, 2D models have demonstrated the existence of viscous critical latitudes for most  inertial modes. 
 
In this article, by employing stream function formulation in a 2D spherical framework, we reduce the vectorial viscous-inertial wave equation to a fourth-order scalar equation of Orr-Sommerfeld type \cite{Gizon2020,Damien22} \hl{for $\Psi$}. Additionally, we derive the separated equations for each longitudinal-wavenumber $m$ with realistic and physically meaningful $m$-dependent boundary conditions at the poles. 
The separated equations serve as the underlying models for our inverse problem: the retrieval of differential rotation \hl{$\Omega$} and turbulent viscosity \hl{$\gamma$ via  the parameter-to-state map $S:(\gamma,\Omega)\mapsto \Psi$ and suitable observation operators $L$ detailed in Subsections \ref{sec:modeling-separated} and \ref{sec:data}, respectively}. Since solar inertial wave inversion is completely novel in helioseismology, this work establishes the first step towards extracting these quantities from observed surface horizontal velocities.

\paragraph{Contributions.} This article provides a comprehensive treatment -- 
modeling, forward problem and inverse problem -- of purely toroidal inertial modes  governed by a time harmonic fourth-order equation on a differentially rotating sphere.

After a derivation of the forward model used in \cite{Damien22}, we prove existence, uniqueness and stability of full and separated inertial wave equations. Well-posedness \hl{of the parameter-to-state map $S$} is guaranteed under explicit conditions that the latitudinal-dependent rotation $\Omega$ is small relative to the product of the frequency and the third power of the viscosity $\gamma$. 
These results enable the application of analytic Fredholm theory, providing a mathematical foundation for the model \cite{Damien22}, and rigorously validating  structure of discrete isolated inertial modes.

We develop a regularization framework for the inverse problem of simultaneously reconstructing viscosity and differential rotation.
The framework is built on characterization of continuous adjoint operators and convergence guarantees for iterative regularization methods -- a result we obtain for both full and, under certain condition, for restricted observations. Central to our convergence analysis is the establishment of the \emph{tangential cone condition} through a lifted regularity strategy, accommodating realistic $L^2$-data. Furthermore, we prove local unique identifiability:  $\Omega$ is uniquely identifiable when $\gamma$ is known, and vice versa under full measurement.

\paragraph{Structure.} The article is organized as follows. Section \ref{sec:equation-data} introduces the modeling framework for inertial oscillations and outlines different observation strategies for solar data. Section \ref{sec:manifold} establish well-posedness and regularity of the resulting wave solutions. 
Section \ref{sec:ip} addresses the inverse problem for viscosity and differential rotation, including sensitivity analysis and adjoint operator derivation. 
In Section \ref{sec:reg_converge}, we develop convergence guarantees for iterative reconstruction methods and prove local unique identifiability via the tangential cone condition. Section \ref{sec:numerics} demonstrate numerical reconstruction performance
across different measurement scenarios. Finally, Section \ref{sec:outlook} concludes our findings  and future  directions.

\section{Model and observation of inertial waves}\label{sec:equation-data}
We begin with deriving a model for dynamics of linear inertial waves on the Sun in a rotating frame, incorporating latitudinal differential rotation and eddy viscosity. The theoretical framework is developed in Sections \ref{sec:modeling}–\ref{sec:modeling-separated}, while observations of solar data are discussed in Section \ref{sec:data}.

\subsection{Reduced-order modeling of inertial waves}\label{sec:modeling} 
\paragraph*{Linearized Navier Stokes equations.} 
We begin with the equation of motion (momentum equation), in which the dynamics of a moving particle of density $\rho$  is subject to linear viscous stress $\mathbf{\boldsymbol\tau}$ with viscosity $\gamma$, acoustic pressure $p$, gravity $\mathbf{g}$, and external source $\mathbf{f}$:
\begin{subequations}\label{eqs-motion}
\begin{align}\label{eq-motion}
\begin{aligned}
&\rho\(\frac{\partial \vv}{\partial  t}+ \covder{\vv}\vv + 2\Orefv\times\vv\)
= \vdiver( \rho \gamma \mathbf{\boldsymbol\tau})
- \grad p 
+ \rho \mathbf{g} 
+ \rho\mathbf{f}, \\
& \mathbf{\boldsymbol\tau}:=\grad\vv+(\grad \vv)^\top.
\end{aligned}
\end{align}
Here, $\vv=\vv(\rv,t)\in\R^3$ is the vector velocity of a particle at position $\rv\in \R^3$ and time $t\in(0,T)$, while $\covder{\vv}:=\sum_{i=1}^3v_i\diffq{}{x_i}$ applies componentwise.
We work in the frame rotating at the constant angular frequency $\Oref$ around a fixed axis $\mathbf{\hat{a}}$ and set $\Orefv=\Oref \,\mathbf{\hat{a}}$. Equation \eqref{eq-motion} corresponds to the Navier-Stokes equation in a rotating frame. The fictitious force $(2\Orefv\times\vv)$ is the inertial Coriolis force resulting from the transformation between the fixed and the rotating frame. The other additional fictitious forces, Euler force $\rho\left(\frac{d\Orefv}{dt}\times\rv\right)$ and centrifugal force $\rho\Orefv\times(\Orefv\times\rv)$, are ignored as  $\Oref$ is uniform in time and assumed to be small. For solar applications, one usually chooses the Carrington frame as reference frame with the angular velocity $\Oref=14.7$~deg/day and the rotation axis $\base{z}=[0,0,1]^{\top}$. Equation~\eqref{eq-motion} is complemented by the continuity equation (conservation of mass):
\begin{equation}
\frac{\partial \rho}{\partial t}+\diver(\rho \vv)=0\,\label{eq-mass}.
\end{equation}
\end{subequations}

We linearize Eqs.~\eqref{eqs-motion} around a stationary background (unperturbed, equilibrium, mean) medium characterized by $\uvn$, $p_0$, $\rho_0$, and $\mathbf{g}_0$ such that
\begin{align*}
&\vv(\rv,t)=\uvn(\rv)+\uv(\rv,t), \quad \rho(\rv,t)=\rho_0(\rv)+\rho'(\rv,t),
\\ &p(\rv,t)=p_0(\rv)+p'(\rv,t),\quad 
\mathbf{g}(\rv,t)=\mathbf{g}_0(\rv)+\mathbf{g}'(\rv,t),
\end{align*}
where $\uv$ is the (perturbed) wave velocity, and 
$\uvn(\rv):= (\Omv(\rv)-\Orefv)\times\rv$ is assumed to be a differential rotation field 
with angular velocity $\Omv(\rv)=\Omega(r,\theta)\base{z}$ depending only 
on radius $r$ and colatitude $\theta$ in spherical coordinates $[r;\theta;\phi]\in [0,\infty)\times [0,\pi)\times [0,2\pi)$. This linearization is possible when no linearly unstable mode exists. In this regime, there is no growing (in time) solutions to Eq.~\eqref{eqs-motion}. Note that such solutions exists in nonlinear simulations with solar-like differential rotation \cite{Bekki2024}, however  this regime is not included in this study.

\paragraph*{Background medium.} We assume that the background medium solves the second-order quasilinear elliptic boundary 
value problem 
\begin{subequations}\label{eqs:order0}
\begin{align}
\label{eq:order0}
&-\vdiver \left(\gamma\rho_0\left(\grad\uvn + (\grad\uvn)^{\top}\right)\right)
+ \rho_0 \covder{\uvn} \uvn + 2\rho_0 \Orefv\times\uvn \nonumber\\
&\hspace{9cm}=-\grad p_0 +\rho_0 \mathbf{g}_0, \\
&\diver (\rho_0\uvn)=0. \label{eq:order0b}
\end{align}
\end{subequations}
The background gravity force derives from a potential $\mathbf{g}_0 = - \nabla \Phi_0$ that satisfies the Poisson equation $\Delta \Phi_0 = 4 \pi G \rho_0$, where $G$ is the gravitational constant. To construct such a background model for the Sun, it is usually assumed that the background flows are weak and the background medium is built from the hydrostatic equilibrium, that is $\grad p_0 = \rho_0 \mathbf{g}_0$. This assumption leads to background coefficients that depend only on depth and are given by standard solar models such as Model~S \cite{modelS}. Starting from a radial model for density, it is possible to solve Poisson equation to obtain the gravitational potential and deduce the background pressure.

\paragraph*{First-order equations.} The first-order the system \eqref{eqs-motion} then becomes
\begin{subequations} \label{eqs:first-order}
\begin{align}
    &\rho_0\(\frac{\partial \uv}{\partial  t}+ \covder{\uvn}\uv +\covder{\uv}\uvn + 2\Orefv\times\uv\)= -\grad p' + \rho' \mathbf{g}_0 + \rho_0 \mathbf{g}' + \vdiver \left[ \rho_0 \gamma \boldsymbol{\tau}' \right]+\rho_0\mathbf{f}, \label{eq:first-order-momentum} \\
    & \frac{\partial \rho'}{\partial t} + \mathrm{div}(\rho' \mathbf{u}_0) + \mathrm{div}(\rho_0 \mathbf{u}) = 0,
\end{align}
\end{subequations}
where $\boldsymbol{\tau}':=\grad \uv+ (\grad \uv)^{\top}$.  It is complemented by the linearized equation of state
\begin{eqnarray*}
    \frac{p'}{p_0} = \gamma \frac{\rho'}{\rho_0} + \frac{s'}{c_v},
\end{eqnarray*}
where $\gamma = 5/3$ is the ratio of specific heats, the quantity $c_v$ is the specific heat at constant volume and $s'$ denotes the perturbed entropy. The first two terms of the right hand side of Eq.~\ref{eq:first-order-momentum} can be rewritten as
\begin{equation}
    -\grad p' + \rho' \mathbf{g}_0 = -\rho_0 \grad \left( \frac{p'}{\rho_0} \right) 
    +p' \hl{\left( - \frac{\grad \rho_0}{\rho_0} + \frac{\grad p_0}{\gamma p_0} \right)}
    , \label{eq:rhs-first-order}
\end{equation}
where we used hydrostatic equilibrium. \hl{The last term is related to the buoyancy frequency $N$ defined as}
\begin{equation}
    N^2 = \mathbf{g}_0 \cdot \left( - \frac{\grad \rho_0}{\rho_0} + \frac{\grad p_0}{\gamma p_0} \right).
\end{equation}

\begin{assumption}
We employ the Cowling approximation with $\mathbf{g}' = 0$ (see \cite{Cowling1941}), which is classic in helioseismology. We further assume the anelastic approximation which is justified if the flow speed is significantly smaller than the sound speed. This is adequate for inertial modes -- flow speeds of a few meters per second compared to kilometers per second for the surface sound speed -- and has been numerically verified \cite{Sup2025}. Two additional assumptions enable reduction to a 2D problem on the sphere: adiabaticity ($s'=0$), and strongly stratified medium (see also \cite{Watson81}) so that the term in $N^2$ in Eq.~\ref{eq:rhs-first-order} can be neglected.

\end{assumption}
Under these assumptions, the system \eqref{eqs:first-order} becomes
\begin{subequations}\label{eqs-motion-1}
\begin{align}
&\frac{\partial \uv}{\partial  t}+ \covder{\uvn}\uv +\covder{\uv}\uvn + 2\Orefv\times\uv= -\grad \left(\frac{p'}{\rho_0} \right) + \frac{1}{\rho_0} \vdiver \left[ \rho_0 \gamma \boldsymbol{\tau}' \right]+\mathbf{f}\label{eq-motion-1}, \\
&\diver(\rho_0\uv)=0, \label{eq-mass-1}
\end{align}
\end{subequations}
where the flow field $\uvn$, equivalently $\Omega$,
and the viscosity $\gamma$ will be the quantities of interest in the inverse problem.

\paragraph*{First-order equations on the sphere.}
We now rewrite the left hand side of \eqref{eq-motion-1} in spherical coordinates using the components 
$u_r, u_{\theta}, u_{\phi}$ in the expansion $\uv=u_r\base{r}+u_{\theta}\base{\theta}´+ u_{\phi}\base{\phi}$, where 
$\base{r}$, $\base{\theta}$, $\base{\phi}$  are the local orthonormal basis vectors defined in the appendix \ref{sec:appendix}. Noting that 
$\uvn=(\Omega(r,\theta)-\Oref)\sin\theta \base{\phi}
$
and using the identities 
\begin{align*}
&\Omv\times\uv = \Omega\left(-\sin\theta \,u_\phi\base{r}-\cos\theta\, u_\phi \base{\theta}+(\sin\theta \,u_r+\cos\theta \,u_\theta)\base{\phi}\right),\\
&\covder{\uvn}\uv=(\Omega-\Oref)(r,\theta)\frac{\partial \uv}{\partial\phi}+(\Omv-\Orefv)\times\uv,\\
&\covder{\uv}\uvn=(\Omv-\Orefv)\times\uv + \left(\diffq{\Omega}r(r,\theta)r\sin\theta\, u_r+\diffq{\Omega}\theta(r,\theta)\sin\theta\, u_\theta\right)\base{\phi},
\end{align*}
we obtain
\begin{align}
&\left[\base{r}; \base{\theta}; \base{\phi} \right]^{\top}
\(\frac{\partial \uv}{\partial  t}+ \covder{\uvn}\uv +\covder{\uv}\uvn + 2\Orefv\times\uv\)
\nonumber\\
\label{eq:PDE_spherical_coo}
=&\mathbf{D}_t\begin{bmatrix}u_r\\ u_{\theta}\\u_{\phi} \end{bmatrix} + 
\begin{bmatrix}
0 &0&-2\Omega\sin\theta\\
0&0& -2\Omega\cos\theta\\
\(2\Omega\sin\theta+ \diffq{\Omega}rr\sin\theta\) & \(2\Omega\cos\theta+\diffq{\Omega}\theta\sin\theta \)&0
\end{bmatrix}\begin{bmatrix}u_r\\ u_{\theta}\\u_{\phi} \end{bmatrix},
\end{align}
where $\mathbf{D}_t := \frac{\partial}{\partial  t}+(\Omega-\Oref)\frac{\partial}{\partial\phi}$ is the \emph{material derivative} acting on vector fields.

Define the $3\times 2$-matrix $\Rco(\rv):=[\base{\theta}(\rv),\base{\phi}(\rv)]$. 
Then multiplication by the $3\times 3$ matrix $\Rco(\rv)\Rco(\rv)^{\top}$ is the projection 
onto the tangential subspace and equivalent to the application of $-\rv\times \rv\times$. 
By applying $-\rv\times \rv\times$ to \eqref{eq-motion-1}, we wish to obtain an approximate equation for the horizontal velocity components 
$\uv_{\mathrm{h}}(\rv):= -\rv\times \rv\times \uv(\rv)$ given by 
$\uv_{\mathrm{h}} = \Rco(\rv)\Rco(\rv)^{\top}\uv= 
u_{\theta}\base{\theta}+ u_{\phi}\base{\phi}$. To this end, we require the following assumption:

\begin{assumption}
Viscosity and density  are radially symmetric, i.e.\ 
$\gamma=\gamma(r)$ and $\rho_0=\rho_0(r)$. 
The fluid is strongly stratified in the sense that the  
radial motion $u_r$ and its derivatives are small compared to the horizontal components $u_\phi,u_\theta$. 
\end{assumption}

Under this assumption, one can neglect the first column of the matrix in Eq.~\eqref{eq:PDE_spherical_coo}, and the application of $\Rco(\rv)\Rco(\rv)^{\top}$ \enquote{kills} the first row of this matrix (the first element of the vector in
\eqref{eq:PDE_spherical_coo} may be balanced by $\diffq{p'}{r}$). 
Therefore, we are left with the matrix 
\[
\Mmat(\rv):= \Rco(\rv) 
\begin{bmatrix}
0& -2\Omega(r,\theta)\cos\theta\\
2\Omega(r,\theta)\cos\theta+\diffq{\Omega}\theta(r,\theta)\sin\theta &0
\end{bmatrix}
\Rco(\rv)^{\top}
\]
and obtain the following system of equations for the tangential vector field 
$\uvh$ on the sphere $\Mc$ for a fixed radius $r>0$ as:
\begin{subequations}\label{eqs:stratified}
\begin{align}\label{eq:stratified}
&-\rho_0\gamma\vLaph\uvh + \rho_0 \mathbf{D}_t\uvh + \rho_0\Mmat\uvh =  \Grad p' + \rho_0 \fvh,
\\
\label{eq-mass-S2}
&\Div \uvh = 0.
\end{align}
\end{subequations}
with the right hand side $\fvh:=\Rco \Rco^{\top} \fv$. Here, we employ the fact $\grad(\rho_0(\grad\uvh)^{\top})=\Grad(\rho_0\Div\uvh)=0$ 
due to \eqref{eq-mass-S2}, with operators $\vLaph$, $\Grad$, $\Div $ defined in \eqref{eq:SurfaceDiffOp}-\eqref{eq:diffops_on_sphere} in Appendix \ref{sec:appendix}. We note that all differential operators with subscript ${\rm h}$ denote operators on spheres and are defined in Appendix \ref{sec:appendix}.

Physically, the assumption of small radial motions is motivated by an identification made in \cite{GizonInertial21}, namely, that the observed inertial modes can be viewed as eigenvalues of the 3D linearized equations of momentum, mass and energy. In this setting, the radial velocities of the eigenfunctions are small relative to the horizontal components for most observed modes. This observation is also supported by nonlinear simulations of rotating convection in a spherical shell \cite{Bekki2022}. While this assumption could potentially be justified through asymptotic analysis ($r\to\infty$ or $\Omega\to 0$), such a rigorous treatment would require amending the system \eqref{eqs-motion-1} with an atmospheric model and explicit assumptions on the coefficient functions -- a task beyond this paper's scope.

\paragraph*{Reformulation as scalar fourth-order equation.}
It follows from the Helmholtz decomposition \eqref{eq:HelmholtzDecomposition} and Eq.~\eqref{eq-mass-S2} that there exits 
a stream function $\Psi$ such that \[\rho_0 \uvh=\vCurl( \Psi  )\] with $\vCurl$ as in \eqref{eq:SurfaceDiffOp}-\eqref{eq:diffops_on_sphere}. 
In order to obtain an equation for $\Psi$, we apply $\Curl$ to \eqref{eq:stratified}. 
Noting that $\Curl (\rho_0\uvh) = \Curl\vCurl\Psi = -\Laph\Psi$ and using the identities 
\begin{align}\label{eq:defi_alpha}
&\Curl \mathbf{D}_t(\rho_0\uvh)=-D_t\Laph\Psi -\frac{1}{r^2}\diffq{\Omega}{\theta}\frac{\partial^2\Psi}{\partial\theta\,\partial\phi}, \nonumber\\
& \Curl\rho_0 \gamma\vLaph \uvh = -\gamma (\Curl\vCurl)^2\Psi = -\gamma \Laph^2\Psi\\
&\Curl\Mmat(\rho_0\uvh) = \alpha_{\Omega}\diffq{\Psi}{\phi} +\frac{1}{r^2}\diffq{\Omega}{\theta}\frac{\partial^2\Psi}{\partial\theta\,\partial\phi},
\quad \alpha_\Omega(\theta):=\frac{1}{r^2\sin\theta}\frac{d}{d\theta}\(\frac{1}{\sin\theta}\frac{d}{d\theta}(\Omega(\theta)\sin^2\theta)\) \nonumber
\end{align} 
with the analog $D_t := \frac{\partial}{\partial  t}+(\Omega-\Oref)\frac{\partial}{\partial\phi}$ 
of $\mathbf{D}_t$, 
we arrive at the scalar fourth-order equation
\begin{equation}\label{eq-vorticity}
\gamma \Laph^2\Psi
-D_t\Laph\Psi 
+ \alpha_\Omega\frac{\partial\Psi}{\partial\phi}   =f
\end{equation}
with $f:=\rho_0\Curl \fvh$, and $\alpha_\Omega$ is a linear function of $\Omega$ as in \eqref{eq:defi_alpha} for convenience.
Lastly, if we are looking for time-harmonic solutions $\Psi(t,\rv) = \Re\left(e^{i\omega t}\Psi_{\omega}(\rv)\right)$, then 
the space-dependent part has to satisfy the fourth-order elliptic equation
\begin{equation}\label{eq-vorticity-freq}
\gamma \Laph^2\Psi 
+i\omega \Laph\Psi -\beta_{\Omega}\Laph\diffq{\Psi}{\phi} 
+ \alpha_\Omega\frac{\partial\Psi}{\partial\phi}
=f
\end{equation}
with $\beta_{\Omega}(\theta):= \Omega(r,\theta)-\Omega_{\rm ref}$. The resulting equation \eqref{eq-vorticity-freq} serves as the underlying model of the forward problem studied in Section \ref{sec:manifold}.

\subsection{Separated equation as inversion model}\label{sec:modeling-separated}
To prepare for the inversion analysis,  we now derive the separated version of \eqref{eq-vorticity-freq} that will serve as our model for parameter identification.
Expanding $\Psi_{\omega}(\theta,\phi)=
\frac{1}{\sqrt{2\pi}}\sum_{m=-\infty}^{\infty}\Psih_{\omega,m}(\theta)e^{im\phi}$ into 
a Fourier series in azimuth, the coefficients $\Psih_{\omega,m}$ 
satisfy the ordinary boundary value problems
\begin{align}\label{eq:separated}
\begin{cases}
\gamma \Lapmh^2\Psih_{\omega,m}(\theta)
+i\omega\Lapmh\Psih_{\omega,m}(\theta)
-im\beta_\Omega(\theta)\Lapm\Psih_{\omega,m}(\theta) 
+im\alpha_\Omega(\theta)\Psih_{\omega,m}(\theta)=\fh_m(\theta)\\
(\Gamma_m \Psih_{\omega,m})(\theta)=0,\qquad \theta\in P:= \{0,\pi\}
\end{cases}
\end{align}
with $\Lapmh$ defined by \eqref{eq:Lap_decomp} in Appendix \ref{sec:appendix}, and 
the following boundary value operators derived in Appendix \ref{appendix:boundary}
\begin{align}\label{bc-operator}
\Gamma_0 \Psi:=\begin{pmatrix}\Psi'|_P\\ \Psi'''|_P \end{pmatrix},\quad 
\Gamma_{\pm 1} \Psi:=\begin{pmatrix}\Psi|_P\\ \Psi''|_P \end{pmatrix},\quad 
\Gamma_m \Psi:=\begin{pmatrix}\Psi|_P\\ \Psi'|_P \end{pmatrix},\; |m|\geq 2. 
\end{align}
We remark that this equation  reduces to Eq.~(2.9) in \cite{Watson81}  if $\gamma=0$ and to Eq.~(8) in \cite{Damien22} if $\fh_m=0$.
However, vanishing Cauchy data are imposed at the poles in \cite{Damien22}, which for $|m|\leq 1$ may yield solutions that do not
correspond to the unseparated equation \eqref{eq-vorticity-freq}. 

The system \eqref{eq:separated}–\eqref{bc-operator}  provides the forward model for the parameter inversion problem addressed in Section \ref{sec:ip}. Solving this inverse problem requires inertial wave observations. The data acquisition process is described in the following section.

\subsection{Observation of solar inertial modes}\label{sec:data}

Solar physicists routinely measure the horizontal flow field $\mathbf{u}_h(\rv,t)$ at the solar surface, in the rotating frame $\Omega_{\rm ref}$ at intermediate spatial scales with a typical temporal cadence of one day \cite{GizonInertial21}. These measurements are obtained using local correlation tracking (LCT), which infers large-scale flows from the motion of small-scale convective structures (granules) observed in intensity images. High-resolution imaging from instruments such as HMI on the SDO spacecraft is required to resolve these granules. The stream function can be computed from 
$
\Psi = -\Delta_h^{-1}  (\Curl(\rho_0 \mathbf{u}_h)),
$
then Fourier transformed in time and longitude to obtain $\Psi^{m,\omega}(\theta)$, which provides data for the inverse problems. 
Due to the limited observation coverage (only half of the Sun is visible at any time), the azimuthal Fourier transform integrates over $[0,\pi]$ instead of $[0,2\pi]$, causing leakage of the signal between the harmonic degrees $m$; see \cite{Basu2016} for a review. We will not concern ourselves with this aspect, but instead focus on the observational limitation: higher latitudes cannot easily be observed due to the line-of-sight projection.
Thus, we assume to observe
\begin{align}\label{data}
y^{\delta}(\theta) :=\Psi^{m,\omega}(\theta)+\mathrm{noise}(\theta), 
\quad \theta\in I:= (0,\pi) \quad\text{or}\quad\theta\in I':=(\epsilon, \pi-\epsilon), \epsilon \geq 0,
\end{align}
where $\delta = \|\mathrm{noise}\|$ denotes the noise level. The case $\epsilon=0$ corresponds to full surface coverage, while $\epsilon > 0$ takes into account the missing data at higher latitudes.

\section{Inertial waves on the manifold $\Mc$}\label{sec:manifold}

\label{sec:well-posed}
With the model established, we now address the forward problem by proving existence, uniqueness and stability of wave solutions, for both the full equation \eqref{eq-vorticity-freq} and its separated counterpart \eqref{eq:separated}. We first show the well-posedness of Eq.~\eqref{eq-vorticity-freq} associated to the differential operator 
\begin{align}\label{B}
\mathbf{B}(\Psi):=\gamma \Laph^2\Psi+i\omega\Laph\Psi-\beta_\Omega
\Laph\frac{\partial\Psi}{\partial\phi} +\alpha_\Omega\frac{\partial\Psi}{\partial\phi}
\end{align}
with recalling that $\beta_{\Omega}(\theta)= \Omega(\theta)-\Omega_{\mathrm{ref}}$ and 
$\alpha_{\Omega}(\theta)=\frac{1}{r^2\sin\theta}\frac{d}{d\theta}\big(\frac{1}{\sin\theta}\frac{d}{d\theta}(\Omega(\theta)\sin^2\theta)\big)$. 

Consider the space $\Hd{2}$ of functions in $H^2(\Mc)$ with zero mean defined in Appendix \ref{sec:appendix}. By partial integration with taking into account symmetry of the bi-Laplacian and Laplacian operators (see Appendix, Lemma \ref{lem:symmetry}), the corresponding sesquilinear form 
$\mathbf{A}: H_{\diamond}^2(\Mc)\times H_{\diamond}^2(\Mc)\to \C$ defined as   
$\mathbf{A}[\Psi, \psi]:=\langle \mathbf{B}(\Psi),\psi\rangle_{H_{\diamond}^{-2},H_{\diamond}^2}$ is given by 
\begin{align}\label{eq:defi_A}
\begin{aligned}
\mathbf{A}[\Psi, \psi]&=\gamma\lsp\Laph\Psi,\Laph\psi\rsp-i\omega\lsp\Grad\Psi, \Grad\psi\rsp
+(\mathbf{A}^{\alpha} + \mathbf{A}^{\beta})[\Psi,\psi],\\
&\mathbf{A}^{\alpha}[\Psi,\psi]:= -\lsp\alpha_{\Omega}\Psi ,\frac{\partial\psi}{\partial\phi} \rsp,\qquad
\mathbf{A}^{\beta}[\Psi,\psi]:= \lsp \beta_{\Omega}\Laph\Psi ,\frac{\partial\psi}{\partial\phi} \rsp.
\end{aligned}
\end{align}
Since the coefficients $\alpha_{\Omega}$ and $\beta_{\Omega}$ are independent of $\phi$, 
this sesquilinear form is block diagonal with respect to the decomposition 
$\Hd{2} = \bigoplus_{m\in Z}^{\infty}V_m$ with the subspaces \[V_m:=H^2_{\diamond}(\Mc;m)\]
introduced in \eqref{eq:Hs_decomp}. In other words, we have
$\mathbf{A}(\Psi_m,\Psi_n)=0$ for $\Psi_m\in V_m$, $\Psi_n\in V_n$ and $m\neq n$.
Hence, the restrictions $\mathbf{A}_m:= \mathbf{A}|_{V_m\times V_m}$ 
fulfill 
\begin{align}\label{eq:A_separation}
\mathbf{A}(\Psi,\psi) =\sum_{m\in\Z}\mathbf{A}_m(\Psi_m,\psi_m)
\end{align} 
if $\Psi = \sum_{m\in\Z}\Psi_m$ and  $\psi = \sum_{m\in\Z}\psi_m$ with 
$\Psi_m,\psi_m\in V_m$. We will also analyze the separated sesquilinear forms given by 
\begin{align}\label{eq:defi_Am}
&\mathbf{A}_m[\Psi_m, \psi_m]\\
&\quad=
\gamma\lan\Lapm\Psi_m,\Lapm\psi_m\ran-i\omega\lan\Gradm\Psi_m, \Gradm\psi_m\ran 
+ (\mathbf{A}^{\alpha}_m+\mathbf{A}^{\beta}_m)[\Psi_m,\psi_m],\nonumber\\
&\mathbf{A}^{\alpha}_m[\Psi_m,\psi_m]
:= im \lsp \alpha_{\Omega}\Psi_m,\psi_m \rsp,\qquad \;
\mathbf{A}^{\beta}_m[\Psi_m,\psi_m]:= -im\lan \beta_\Omega\Lapm\Psi_m,\psi_m \ran. \nonumber
\end{align}
with $\mathbf{A}^{\alpha}_m, \mathbf{A}^{\beta}_m:V_m\times V_m\to \C$ as restrictions of $\mathbf{A}^{\alpha}$, $\mathbf{A}^{\beta}$  defined in \eqref{eq:defi_A}. Here, $\Lapm$, $\Gradm$ are the decomposed operators as defined in \eqref{eq:Lap_decomp}, \eqref{eq:Grad_decomp}.

\subsection{Boundedness}\label{sec:boundedness}
It is straightforward to verify that $\mathbf{A}^{\alpha}$ and $\mathbf{A}^{\alpha}_m$ are bounded if 
$\Omega\in W^{2,\infty}(0,\pi)$. However, to reduce the regularity assumption on the unknown rotation, we only constrain 
\begin{align}\label{eq:Omega_regularity}
\Omega\in H^1(\Mc; m=0)
\end{align}
\hl{to obtain well-posedness of wave solutions in $H^2$. For higher regularity of $\Psi$, as discussed in later sections, stronger assumption on $\Omega$ will be introduced. For now, we focus on boundedness results.}
Under this reduced regularity assumption, establishing boundedness requires additional technical analysis. 
In the following, we will frequently employ continuous embeddings between function
spaces $X,Y$ denoted the embedding constants $C_{X\to Y}$, i.e., 
$\|\cdot\|_{Y}\leq C_{X\to Y} \|\cdot\|_X$.
\begin{lemma}\label{lem:Aalpha}
 If $\Omega\in C^2([0,\pi])$, then 
for any $m\in\Z$ and $\Psi_m,\psi_m\in V_m\cap C^{\infty}(\Mc)$ we have
\begin{align}\label{eq:alt_defi_Aalpha}
\mathbf{A}^{\alpha}[\Psi_m,\psi_m]&= 
\lsp \tilde{\alpha}_{\Omega}\Grad \Psi_m,\vCurl\psi_m\rsp
\end{align}
where $\tilde{\alpha}_{\Omega}(\theta):= \Omega'(\theta)\sin\theta+2\Omega(\theta)\cos\theta$. Moreover,
\begin{align}\label{eq:Aalpha_bound}
\left|\mathbf{A}^{\alpha}[\Psi_m,\psi_m]\right|
\leq C \|\Omega\|_{H^1}\|\Psi_m\|_{H^2}\|\psi_m\|_{H^{3/2}}
\leq C \|\Omega\|_{H^1}\|\Psi_m\|_{H^2}\|\psi_m\|_{H^2}
\end{align}
where $C:=3C_{H_{\diamond}^1\to L^6}C_{H_{\diamond}^{1/2}\to L^3}$. 

Under the regularity assumption \eqref{eq:Omega_regularity}, the right hand side of \eqref{eq:alt_defi_Aalpha} has a unique 
continuous extension to a bounded sesquilinear form on $H^2_{\diamond}(\Mc)$. And the inequality \eqref{eq:Aalpha_bound} is satisfied for all $\Psi,\psi\in H^2_{\diamond}(\Mc)$, that is,
\begin{align}\label{eq:Aalpha_bound-full}
\left|\mathbf{A}^{\alpha}[\Psi,\psi]\right|
\leq C \|\Omega\|_{H^1}\|\Psi\|_{H^2}\|\psi\|_{H^{3/2}}
\leq C \|\Omega\|_{H^1}\|\Psi\|_{H^2}\|\psi\|_{H^2}.
\end{align}
\end{lemma}

\begin{proof}
Let $\Psi_m,\psi_m\in V_m\cap C^{\infty}(\Mc)$. Note that $\alpha_{\Omega}(\theta) = \frac{1}{r^2\sin\theta}\frac{d}{d\theta}\tilde{\alpha}_{\Omega}(\theta)$. 
As $\tilde{\alpha}_{\Omega}\in C^1([0,\pi])$, 
$\Psi_0(\theta)\frac{\partial\overline{\psi_0}}{\partial \phi}(\theta)=0$ and 
$\Psi_m(\theta)\frac{\partial\overline{\psi_m}}{\partial \phi}(\theta)=-im\Psi_m(\theta)\overline{\psi_m}(\theta)=0$ 
for $\theta\in \{0,\pi\}$,
$m\neq 0$, we can perform a partial integration in $\theta$ to obtain 
\begin{align}\label{Aalpha}
&\mathbf{A}^{\alpha}[\Psi_m,\psi_m] \nonumber\\
&=\int_0^{2\pi}\!\int_0^\pi
\tilde{\alpha}_{\Omega} \frac{\partial}{\partial\theta}
\(\Psi_m\frac{\overline{\psi_m}}{\partial\phi}\)\,d\theta\,d\phi=\int_0^{2\pi}\!\int_0^\pi\tilde{\alpha}_{\Omega}\( \Psi_m\frac{\partial^2\overline{\psi_m}}{\partial\phi\partial\theta}
+\frac{\partial\Psi_m}{\partial\theta}
\frac{\partial\overline{\psi_m}}{\partial\phi} \)\,d\theta\,d\phi \nonumber\\
&=\int_0^{2\pi}\!\int_0^\pi\tilde{\alpha}_{\Omega} \( 
- \frac{\partial\Psi_m}{\partial\phi} \frac{\partial\overline{\psi_m}}{\partial\theta}
+\frac{\partial\Psi_m}{\partial\theta} \frac{\partial\overline{\psi_m}}{\partial\phi}
\) \,d\theta\,d\phi\\
&= \int_{\Mc} 
\frac{\tilde{\alpha}_{\Omega}(\theta)}{r^2 \sin\theta}\( 
- \frac{\partial\Psi_m}{\partial\phi} \frac{\partial\overline{\psi_m}}{\partial\theta}
+\frac{\partial\Psi_m}{\partial\theta} \frac{\partial\overline{\psi_m}}{\partial\phi}
\)\,ds, \nonumber
\end{align}
where in the second line we have performed a further partial integration in $\phi$. 
By using \eqref{eq:diffops_on_sphere}, we then achieve \eqref{eq:alt_defi_Aalpha}. Next, the generalized H\"older's inequality 
\begin{align}\label{eq:Hoelder_triple}
\left|\int_{\Mc}abc\,ds\right|\leq \|a\|_{L^2}\|b\|_{L^3}\|c\|_{L^6}
\end{align} 
implies
\[
\left|\mathbf{A}^{\alpha}[\Psi_m,\psi_m]\right|
\leq \|\tilde{\alpha}_{\Omega}\|_{L^2} \|\Grad \Psi_m\|_{L^3} \|\Grad \psi_m\|_{L^6}.
\]
The bound $\|\tilde{\alpha}_{\Omega}\|_{L^2}\leq \|\frac{\partial\Omega}{\partial \theta}\|_{L^2}+ 2\|\Omega\|_{L^2}\leq 3\|\Omega\|_{H^1}$ together with
the continuity of the embeddings $H^1(\Mc)\hookrightarrow L^6(\Mc)$ 
and $H^{1/2}(\Mc)\hookrightarrow L^3(\Mc)$ in \eqref{eq:embedding_Lp} yields \eqref{eq:Aalpha_bound}.

Since $\mathbf{A}^{\alpha}$ separates in $V_m$ and the spaces $V_m$ are orthogonal in $H^2_{\diamond}(\Mc)$, the bound \eqref{eq:Aalpha_bound} also holds true for finite linear combinations of 
$\Psi_m$ and $\psi_m$ of the form above. The last statement \eqref{eq:Aalpha_bound-full} follows from the density 
of such linear combinations (or even linear combinations of spherical hamonics) in $H^2_{\diamond}(\Mc)$. 
\end{proof}

\begin{proposition}\label{prop:boundedness}
Suppose that $\Omega$ satisfies \eqref{eq:Omega_regularity}. Then 
the sesquilinear form $\mathbf{A}$ defined by \eqref{eq:defi_A} 
is bounded on $H^2_{\diamond}(\Mc)$, and the separated bilinear forms $\mathbf{A}_m$ defined by \eqref{eq:defi_Am} are 
bounded on $V_m$. Both $\mathbf{A}$ and $\mathbf{A}_m$ depend continuously on $\Omega$ 
with respect to its $H^1$-norm. 
\end{proposition}

\begin{proof}
Boundedness of $\mathbf{A}^{\alpha}$ has been shown in Lemma \ref{lem:Aalpha}. For $\mathbf{A}^{\beta}$, 
we again use the H\"older inequality \eqref{eq:Hoelder_triple} and the 
same Sobolev embeddings as above to obtain
\begin{equation}\label{eq:Abeta_bound}
\begin{split}
|\mathbf{A}^{\beta}(\Psi,\psi)|&\leq \|\beta\|_{L^6}\|\Delta_{\rm h}\Psi\|_{L^2}\|\Grad\psi\|_{L^3}\\
&\leq C_{H^1\to L^6} C_{H^{1/2}\to L^3}
\|\Omega-\Oref\|_{H^1}\|\Psi\|_{H^2}\|\psi\|_{H^{3/2}}.
\end{split}
\end{equation} 
Boundedness of the first two terms in \eqref{eq:defi_A} is obvious, and boundedness of 
$\mathbf{A}_m$ follows from that of $\mathbf{A}$, using \eqref{eq:A_separation} and the pairwise orthogonality of the spaces $V_m$. 
\end{proof}

\subsection{Fredholm property and well-posedness}\label{sec:wellposed}
\begin{theorem}\label{theo:non-separate-eq} 
Let $\gamma>0$ and $\Omega\in H^1(\Mc; m=0)$ as in Proposition \ref{prop:boundedness}.
\begin{enumerate}
\item 
The differential operator $\mathbf{B}:\Hd{2}\to \Hd{-2}$ defined in \eqref{B}
is Fredholm of index $0$.
\item Moreover, $\mathbf{B}$ is boundedly invertible at sufficiently large frequencies satisfying 
\begin{align}\label{non-separated-omega}
|\omega|>\frac{4}{\gamma^{3}}(C_{H_{\diamond}^1\to L^6}C_{H_{\diamond}^{1/2}\to L^3})^4\big(\|\Omega-\Oref\|^2_{H^1}+ 9\|\Omega\|^2_{H^1} \big)^2.
\end{align}
\end{enumerate}
\end{theorem}

\begin{proof}
\emph{Part 1:} 
Applying Young's inequality $ab\leq \frac{2}{\gamma}a^2+\frac{\gamma}{8} b^2$ 
to the bounds \eqref{eq:Aalpha_bound} and \eqref{eq:Abeta_bound} 
on $|\mathbf{A}^{\alpha}|$ and $|\mathbf{A}^{\beta}|$ 
yields the following $\mathrm{G\mathring{a}rding}$-type inequality:
\begin{align*}
&\sqrt{2}|\mathbf{A}[\Psi, \Psi]|
\geq |\Re\{\mathbf{A}[\Psi, \Psi]\}|+|\Im\{\mathbf{A}[\Psi, \Psi]\}|\\[1ex]
 &\geq\gamma\|\Laph\Psi\|^2_{L^2}+|\omega| \|\Grad\Psi\|^2_{L^2}
 -|\mathbf{A}^{\alpha}[\Psi,\Psi]|-|\mathbf{A}^{\beta}[\Psi,\Psi]|\\
&\geq \frac{3\gamma}{4}\|\Psi\|^2_{H_{\diamond}^2}
+|\omega| \|\Psi\|^2_{H_{\diamond}^1}
- \frac{2}{\gamma}(C_{H_{\diamond}^1\to L^6}C_{H_{\diamond}^{1/2}\to L^3})^2\big(\|\Omega-\Oref\|^2_{H_{\diamond}^1}+ 9\|\Omega\|^2_{H_{\diamond}^1} \big)\|\Psi\|^2_{H_{\diamond}^{3/2}}. 
\end{align*}
Since the embeddings \hl{$ \Hd{2}\embed \Hd{1}$} and \hl{$ \Hd{2}\embed\Hd{3/2}$}  are compact, 
$\mathbf{B}$ is a Fredholm operator of index $0$; see, e.g., \cite[Sec.~8.2.4]{Renardy1992}.

\emph{Part 2:}  Applying the interpolation inequality $\|\Psi\|_{H_{\diamond}^{3/2}}^2\leq \|\Psi\|_{H_{\diamond}^1}\|\Psi\|_{H_{\diamond}^2}$ in the last inequality and 
then Young's inequality   $ab\leq \frac{1}{\gamma}a^2+\frac{\gamma}{4} b^2$ yields
\begin{align*}
&\sqrt{2}|\mathbf{A}[\Psi, \Psi]|\\&\quad\geq \frac{\gamma}{2}\|\Psi\|^2_{H_{\diamond}^2}
+\(|\omega| \!-\! \frac{4}{\gamma^{3}}(C_{H_{\diamond}^1\to L^6}C_{H_{\diamond}^{1/2}\to L^3})^4
\big(\|\Omega\!-\!\Oref\|^2_{H^1}\!+ \!9\|\Omega\|^2_{H^1} \big)^2 \)\|\Psi\|^2_{H_{\diamond}^1}.
\end{align*}
If $|\omega|$ satisfies \eqref{non-separated-omega}, then $\mathbf{A}$ is coercive, 
so $\mathbf{B}$ is boundedly invertible by the Lax-Milgram theorem \cite[Section 6.2, Theorem 1]{Evans}. 
\end{proof}

We now indicate the dependence of $\mathbf{B}$ on the frequency $\omega$ explicitely by $\mathbf{B}_{\omega}$ and also consider complex values of $\omega$. This allows us  
to apply analytic Fredholm theory and study resonances: 
\begin{corollary}\label{coro:analyticFredholm}
There exists a discrete set $\Lambda\subset \mathbb{C}$ without accumulation points such that $\mathbf{B}_{\omega}$ is boundedly invertible 
in $L(H_{\diamond}^2(\Mc),H_{\diamond}^{-2}(\Mc))$ for $\omega\in \mathbb{C}\setminus\Lambda$, and $\dim\operatorname{ker}\mathbf{B}_{\omega}\in \mathbb{N}$ for $\omega\in \Lambda$. 
\end{corollary}
The elements of $\operatorname{ker}\mathbf{B}_{\omega}$ for $\omega\in\Lambda$ are called inertial modes and 
have been studied intensively in \cite{Damien22} and other publications. 
\begin{proof}[Proof of Corollary \ref{coro:analyticFredholm}]
It follows from part 1 in the proof of Theorem \ref{theo:non-separate-eq} that $\mathbf{C}(\omega):H^2(\Mc)\to H^2(\Mc)$  defined via
$\mathbf{C}(\omega):=(I-\Delta_{\mathrm{h}})^{-2}\mathbf{B}_{\omega}-I$ is a compact operator for all $\omega\in\mathbb{C}$. Moreover, $\omega\mapsto \mathbf{C}(\omega)$ is affinely linear and in particular holomorphic.

 Therefore, we are in the framework of the analytic Fredholm theorem
(see \cite[Thm.~7.92]{Renardy1992}). Note that $\mathbf{B}_{\omega}$ 
is boundedly invertible if and only of if $I+\mathbf{C}(\omega)$ is boundedly invertible and that the kernels 
of both operators coincide. The first alternative of this 
theory, that is $(I+\mathbf{C}(\omega))^{-1}$ does not exist for any $\omega\in\mathbb{C}$, is excluded by the second 
part of Theorem \ref{theo:non-separate-eq}. Therefore, the second alternative holds true and provides the 
statements of this corollary. 
\end{proof}

\subsection{Separated equations}\label{sec:wellposed-separate}

Let us summarize some consequences for the separated operators 
$\mathbf{B}_m: V_m\to V_m^* = H^{-2}_{\diamond}(\Mc;m)$ defined by 
\begin{align}\label{eq:B_separation}
\lsp \mathbf{B}_m\Psi_m,\psi_m\rsp =   \mathbf{A}_m[\Psi_m,\psi_m],\qquad \Psi_m,\psi_m\in V_m.
\end{align}

\begin{corollary}
For any parameters $\omega,\gamma,\Omega$ for which $\mathbf{B}$ is boundedly invertible, 
in particular for sufficiently large $|\omega|$ satisfying \eqref{non-separated-omega} as in Theorem \ref{theo:non-separate-eq}, 
the separated operators $\mathbf{B}_m:V_m\to V_m^*$ are boundedly invertible for all $m$. 
\end{corollary}

\begin{proof}
It follows from \eqref{eq:A_separation} that $\|\mathbf{B}_m\|\leq \|\mathbf{B}\|$, so 
$\mathbf{B}_m$ is bounded. 
If $\mathbf{B}$ is boundedly invertible, then if follows from \eqref{eq:A_separation} 
and the inf-sup characterization of $\|\mathbf{B}_m^{-1}\|$ that 
$\|\mathbf{B}_m^{-1}\|\leq \|\mathbf{B}^{-1}\|$, 
and $\mathbf{B}_m$ is boundedly invertible as well. This completes the proof. 
\end{proof}

\begin{remark}
Repeating the argument of the proof of Theorem \ref{theo:non-separate-eq}, it can also be 
shown that the separated operators $\mathbf{B}_m$ are Fredholm of index $0$. 
Moreover, Corollary \ref{coro:analyticFredholm} holds true with $\mathbf{B}_{\omega}$ replaced 
by $\mathbf{B}_m=\mathbf{B}_{m,\omega}$. 
\end{remark}

\begin{proposition}[Uniqueness for small $\Omega'$]\label{prop:existence}
For any $m\in\Z$ the operator $\mathbf{B}_m:V_m\to V_m^*$ is boundedly invertible if the derivative $\Omega'$ (not necessarily differential rotation $\Omega$ itself) is small in comparison to the viscosity $\gamma$, in the sense that
\begin{align}\label{smallness}
\|\Omega'\|_{L^2(\Mc)}\frac{|m|}{r}C_{H^2_{\diamond}\to L^3_{\diamond}}C_{H^1_{\diamond}\to L^6_{\diamond}}<\gamma.
\end{align}
\end{proposition}

\begin{proof}
Since $\mathbf{A}_m$ is bounded by Proposition \ref{prop:boundedness}, the Lax-Milgram lemma reduces the proof to verifying the coercivity of $\mathbf{A}$. To this end, we estimate its real part under the smallness condition \eqref{smallness}.
We will employ the (equivalent) norm  $\|u\|_{\Hdm{2}}:=\|\Delta u\|_{L^2(\Mc)}$ for  $u\in \Hdm{2}$. As $\Re \mathbf{A}^{\alpha}_m[\Psi,\Psi]=0$, one obtains
\begin{align*}
\Re\mathbf{A}_m[\Psi, \Psi]&=\gamma\|\Lapm\Psi\|^2_{L^2(\Mc)}
-\Re\big( im\big\langle \beta_\Omega\Lapm\Psi,\Psi \big\rangle\big).
\end{align*}
Using partial integration, the identity 
$\Gradm\beta_{\Omega}= \frac{1}{r}\Omega'\base{\theta}$,
H\"older's inequality \eqref{eq:Hoelder_triple} and the embeddings 
\eqref{eqs:separated_embeddings_sphere}, 
we can bound
\begin{align*}
&\Re  \big(im\big\langle \beta_\Omega\Lapm\Psi,\Psi \big\rangle\big)\\
&\quad\leq |m| \left|\Im \big\langle \Gradm \Psi,\Gradm (\beta_{\Omega}\Psi)\big\rangle\right|
=|m|\left|\Im \big\langle \Gradm\Psi,\Gradm(\beta_{\Omega})\Psi\big\rangle\right|\\
&\quad\leq |m| \|\Gradm \Psi\|_{L^6} \|\Gradm\beta_{\Omega}\|_{L^2} \|\Psi\|_{L^3}
\leq \frac{|m|}{r}C_{H^2_{\diamond}\to L^3_{\diamond}}C_{H^1_{\diamond}\to L^6_{\diamond}}\|\Omega'\|_{L^2(\Mc)}\|\Psi\|_{H^2_{\diamond}}^2.
\end{align*}
Therefore, $\Re\mathbf{A}_m[\Psi, \Psi]\geq c\|\Lapm\Psi\|^2_{L^2(\Mc)}$ 
with $c:=\gamma - \frac{|m|}{r}C_{H^2_{\diamond}\to L^3_{\diamond}}C_{H^1_{\diamond}\to L^6_{\diamond}}\|\Omega'\|_{L^2}>0$ ensured by the assumption \eqref{smallness}. This completes the proof of coercivity.
\end{proof}

\shorten{
\begin{remark}[Pointwise evaluation]
The embedding $(\Psi\in) \Hdm{2}\embed C(\Mc)$ enables the pointwise evaluation of the wave field $\Psi$ everywhere on $\Mc$. This property has been employed in Proposition \ref{prop:Fcov-adjoint}.
\end{remark}

\begin{remark}[Multi azimuth-frequency]
We  keep in mind that here, $\Psi$ is a single mode, i.e~$\Psi=\widehat{\Psi}_{\omega,m}$.
Thus, for inversion using multiple azimuths and/or frequencies, such as a Kaczmarz scheme \cite{HaltmeierLeitaoScherzer}, one should take into account \eqref{smallness}-\eqref{largeness} in a uniform sense.
\end{remark}

\begin{remark}[Helioseismology]
In helioseismology, the rotation rate is of order 400~nHz, the viscosity at the surface is around 100~km$^2$/s and the observed modes have longitudinal wave numbers $m$ smaller than 10 \cite{GizonInertial21}. Verifying the smallness hypothesis in Proposition~\ref{prop:existence} for real Sun applications is an interesting task.
\end{remark}
}

\subsection{Additional regularity}\label{sec:regularity}
We have established unique existence of wave solutions. However, we proceed further with proving higher regularity of the wave solutions, particularly for the separated equations. This regularity result is not merely a technical refinement; it is essential for convergence guarantee of iterative regularization methods later in Section \ref{sec:reg_converge}. \hl{To this end, we need to further require that $\Omega$ be in $H^2$, rather than just in $H^1$, as was assumed in \eqref{eq:Omega_regularity} and the subsequent well-posedness results.}

\begin{proposition}[Lifted regularity]\label{prop:regularity}
Let $m\in \Z$ and suppose  that \eqref{non-separated-omega} or \eqref{smallness} hold true such
that $\mathbf{B}_{m}:\Hdm{2}\to \Hdm{-2}$ is boundedly invertible. 
Further assume that source and  rotation have regularity  
\begin{align}\label{ass:regularity-coe}
f_m\in L^2(\Mc;m), \qquad \Omega\in H^2(\Mc;m=0).
\end{align}
Then the solution $\Psi_m:=\mathbf{B}_m^{-1}f_m$ 
as well as $\tilde{\Psi}_m:=(\mathbf{B}_m^\star)^{-1}f_m$ with 
the adjoint operator  
$\mathbf{B}_m^{\star}:\Hdm{2}\to \Hdm{-2}$
have higher regularity
\begin{align}\label{ass:regularity-state}
\Psi_m\in \Hdm{4}.
\end{align}
Moreover, with $\Psi_m(\theta,\phi) = \hat{\Psi}_m(\theta)e_m(\phi)$, 
the boundary values $\Gamma_m\hat{\Psi}_m$ are well defined 
for all $m\neq 0$ and 
$
\Gamma_m\hat{\Psi}_m=0,
$ 
and similarly for $\tilde{\Psi}_m$. \\
Finally, the restrictions of $\mathbf{B}_m$ and $\mathbf{B}_m^\star$ 
have the forms
\[
\left.\begin{array}{l}\mathbf{B}_m
=\gamma \Lapm^2
+i\omega\Lapm
-im\beta_\Omega\Lapm
+im\alpha_\Omega\\
\mathbf{B}_{m}^\star
=\gamma \Lapm^2
-i\omega\Lapm
+im\Lapm(\beta_\Omega\cdot)
-im\alpha_\Omega
\end{array}
\right\}:\Hdm{4}\to L^2_{\hl{\diamond}}(\Mc;m)
\]
and they have bounded inverses $\mathbf{B}_m^{-1},(\mathbf{B}_m^\star)^{-1}:L^2_{\hl{\diamond}}(\Mc;m)\to \Hdm{4}$.
\end{proposition}

\begin{proof}
The solution $\Psi_m\in \Hdm{2}$ satisfies 
\begin{align}\label{liftref-source}
\begin{aligned}
\forall \psi_m\in \Hdm{2}: \quad& \gamma\lsp \Lapm\Psi_m,\Lapm\psi_m\rsp
= \lsp \tilde{f}_m,\psi_m\rsp\\
& \mbox{with }
\tilde{f}_m := f_m + i\omega \Lapm \Psi_m - im \beta_{\Omega}\Lapm\Psi_m+im\alpha_{\Omega}\Psi_m.
\end{aligned}
\end{align}
We shall claim that $\tilde{f}_m\in L^2(\Mc;m)$. As $\alpha_{\Omega} = \Omega''\sin + 3 \Omega' \cot - 2\Omega\sin$, hence $\Omega\in H^2(\Mc;0)$ implies
\begin{align*}
\|\alpha_\Omega\Psi_m\|_{L^2}&\leq\| (\Omega''- 2\Omega\sin)\Psi_m \|_{L^2}+ \left\|3 \Omega' \cos \frac{\Psi_m}{\sin}\right\|_{L^2}\\
&\leq (\|\Omega''\|_{L^2}+2\|\Omega\|_{L^2})\|\Psi_m\|_{L^\infty} + 3\|\Omega'\|_{L^4}\|\Gradm\Psi_m\|_{L^{4}}
\end{align*}
with noting that $\Psi_m/\sin$ is $\phi$-component of $\Gradm\Psi_m$  as defined in \eqref{eq:Grad_decomp} for $m\neq 0$; 
for the case $m=0$, the term $\alpha_\Omega$ does not appear in the equation.
Using the embeddings \eqref{eqs:separated_embeddings_sphere} and setting 
$C_{\alpha}:= 2C_{H^2\to L^\infty}+3C^2_{H^1\to L^4}$, we obtain
\begin{align}\label{eq:alpha_bound}
\|\alpha_\Omega\Psi_m\|_{L^2}\leq C_{\alpha}\|\Omega\|_{H^2} \|\Psi_m\|_{H^2}.
\end{align}
With this we can derive the following norm bound:
\begin{align*}\label{liftref-source-bounded}
\|\tilde{f}_m\|_{L^2}
&\leq \|f_m\|_{L^2}+|\omega|\|\Psi_m\|_{H^2}+|m|\|\beta_{\Omega}\|_{L^\infty}\|\Psi_m\|_{H^2}+|m|C_\alpha\|\Omega\|_{H^2}\|\Psi_m\|_{H^2}\\
&\leq \|f_m\|_{L^2}+\big(|\omega|+|m|(C_{H^2\to L^\infty}\|\Omega\|_{H^2}+\|\Oref\|_{L^\infty}+C_{\alpha}\|\Omega\|_{H^2})
\big)\|\Psi_m\|_{H^2}. \nonumber
\end{align*}
It follows from \hl{the special elliptic regularity result} \eqref{eq:Laph_isometric_iso}, \hl{which implies that $\Laph^2:\Hd{s}\to \Hd{s-4}$ is bijective and isometric for $s=2,4$,} 
that weak solutions $\Psi\in \Hd{2}$ to 
the biharmonic equation $\Laph^2\Psi =f$ with $f\in L^2_{\hl{\diamond}}(\Mc)$ belong to $\Hd{4}$. Due to the separability of $\Laph$, weak solutions $\Psi_m\in\Hdm{2}$  to the separated equation $\Lapm^2\Psi_m = \tilde{f}_m$ with $\tilde{f}_m\in L^2_{\diamond}(\Mc)$ also belong to $\Hdm{4}$ with $\|\Psi_m\|_{H^4} = \|\tilde{f}_m\|_{L^2}$. 
Together with $\|\Psi_m\|_{H^2}\leq 
\|\mathbf{B}_m^{-1}\| \|f_m\|_{H^{-2}} \leq\|\mathbf{B}_m^{-1}\| \|f_m\|_{L^2}$ 
this shows  
that $\|\Psi_m\|_{H^4}\leq C\|f\|_{L^2}$, i.e., 
$\mathbf{B}_m:\Hdm{4}\to L^2(\Mc;m)$ is boundedly invertible. 
The proof of the analogous statement for $\mathbf{B}_m^\star$ 
only requires a different treatment of the term involving $\beta_{\Omega}$.  
Here, we use the fact that $H^2(\Mc)$ is a Banach algebra 
(see \cite[p.~115]{adams:75}  or \cite[Thm 1.4]{BadrBernicotRuss}) to show that 
$\|\beta_{\Omega}\Psi_m\|_{H^2}\leq C \|\beta_{\Omega}\|_{H^2}\|\Psi_m\|_{H^2}$. 

Concerning the boundary values, we use the continuous embedding $H^4(\Mc)\hookrightarrow C^{2}(\Mc)$ in \eqref{eq:embedding_Ck} and Lemma \ref{lemm:derivatives_at_poles} to 
show that $\Gamma_m\hat{\Psi}_m$ is well-defined and vanishes.  
\end{proof}

\subsection{Well-posedness on the restricted domain}\label{sec:wellposed-restricted}
Motivated by the practical measurement scenario in Section \ref{sec:data}, where observations are limited to a restricted latitude range $I'=(\epsilon,\pi-\epsilon)$, this section derives well-posedness and regularity of the wave operators on this restricted domain. These results are subsequently employed in Section \ref{sec:tcc-restricted} for convergence analysis of reconstruction methods under restricted observations.

We outline how an analysis analogous to that of Propositions \ref{prop:existence} and \ref{prop:regularity} can be conducted. The key elements are symmetry of the bi-Laplacian and Laplacian operators and strategic use of partial integrations. According to Remark \ref{rem:symmetry-restricted}, symmetry of Laplacian operators on function spaces supported on $I'$ is guaranteed with homogeneous Cauchy boundary data, i.e. $\Gamma_2$ in \eqref{bc-operator}. 
To this end, we define the following function spaces on $I'$ with 
norms inherited from $H^s(\Mc;m)$:
\begin{equation}\label{space-I}
\begin{split}
&H^s(I'; m):=\left\{ \Psi|_{I'}: \Psi \in H^s(\Mc;m) \right\}, s\geq 0, \\
&H^s_0(I'; m):=\left\{ \Psi|_{I'}: \Psi \in H^s(\Mc;m), \Psi|_{\partial I'}=\Psi'|_{\partial I'}=0 \right\}, s\geq 2.
\end{split}
\end{equation}
Note that these spaces coincide with the standard Sobolev spaces 
$H^s(I')$ and $H^s_0(I')$ with equivalent norms. Correspondingly, the dual spaces
are, respectively, given by $H^{-s}_0(I';m):=H^{s}(I';m)'=H^{-s}_0(I')$ and 
$H^{-s}(I';m):=H^{s}_0(I';m)'=H^{-s}(I')$, again with equivalent norms.

\begin{proposition}[Well-posedness on $I'$]\label{prop:wellposed-partial}
Given $\gamma>0$, and $\Omega\in H^1(\Mc; m=0)$ satisfying \eqref{non-separated-omega} or \eqref{smallness}.  Assume  for any $m\in\Z$ the Cauchy boundary condition
\begin{align}\label{partial-extra-info}
\Psi|_{\partial I'}=u_1\in\C^2\quad\text{and}\quad \Psi'|_{\partial I'}=u_2\in\C^2.
\end{align}
Let $u\in C^\infty(I)$ with $u|_{\partial I'}=u_1$, $u'|_{\partial I'}=u_2$ and $g:=\mathbf{B}_mu$. We claim:
\begin{enumerate}[label=(\roman*)]
\item
The operator $\mathbf{B}_m:H_0^2(I';m)\to H^{-2}(I';m)$ is boundedly invertible. The wave solution is $\Psi=\mathbf{B}_m^{-1}(f_m-g)+u\in H^2(I';m)$.
\item
Furthermore, if $\Omega\in H^2(\Mc; m=0)$ as in Proposition \ref{prop:regularity},  then the  linear bounded operator $\mathbf{B}_m$, $\mathbf{B}_m^\star$ have bounded inverses $\mathbf{B}_m^{-1},(\mathbf{B}_m^\star)^{-1}:L^2(I';m)\to H^4_0(I'; m)$. \label{prop:wellposed-partial-liftedreg}
\end{enumerate}
\end{proposition}
\begin{proof} The splitting $\Psi=\mathbf{B}_m^{-1}(f_m-g)+u$ allows us to study the differential operators $\mathbf{B}_m$ on the state space $H^2_0(I'; m)$ with homogeneous boundaries as defined in \eqref{space-I}.
Symmetry of the  bi-Laplacian and Laplacian operators with homogeneous Cauchy boundary condition \eqref{partial-extra-info} shown in Remark \ref{eq:defi_A} enables the sequilinear form $\mathbf{A}$ in \eqref{eq:defi_A}. Next, the integration-by-part \eqref{Aalpha} in Lemma \ref{lem:Aalpha} holds under Cauchy boundary condition, implying  boundedness of $\mathbf{A}$, $\mathbf{A}^\alpha$ as in Proposition \ref{prop:boundedness}, and well-posedness of the unseparated equation in the same setting as Theorem \ref{theo:non-separate-eq}. Then, for the separated equations, following the proof of Proposition \ref{prop:existence}, in which the partial integration also holds on the subdomain $I'$, we obtain the same uniqueness result.

Concerning the lifted regularity result as in Proposition \ref{prop:regularity}, by setting the auxiliary source $\tilde{f}_m$ as in \eqref{liftref-source} and partial integration, we (formally) have
$\gamma\lsp \Lapm^2\Psi_m,\psi_m\rsp=\gamma\lsp \Lapm\Psi_m,\Lapm\psi_m\rsp
= \lsp \tilde{f}_m,\psi_m\rsp$ for all $\psi_m\in H^2_0(I'; m)$. 
Using the fact the weak solutions of the bi-Laplace equations are also 
strong solutions, one obtains boundedness of the fourth order term $\|\Delta_m^2\Psi_m\|_{L^2}$, given higher regularity of $\Omega$ and $f$ as in Proposition \ref{prop:regularity}. Alternatively, one can apply general elliptic regularity results as, e.g., in \cite[Sec.~5.3.4]{Triebel}.
\end{proof}

\section{Inverse problem for viscosity and rotation }\label{sec:ip}
With the forward problems established, we now address the inverse problem: recovering the differential rotation function $\Omega$ and the viscosity parameter $\gamma$ in the separated equations under various data measurement strategies.

\subsection{Inversion with different observation strategies}\label{sec:ip-formulation}
As described in Section \ref{sec:data}, data is acquired under two scenarios: full or partial measurements of the wave solution. The measurement operator is hence defined as
\begin{align}\label{Llinear}
\begin{array}{l}
L: V_m\to Y \\
(L\Psi)(\theta):=\Psi(\theta)
\end{array}
\qquad Y=
\begin{cases} L^2(I,r^2\sin),& \text{full measurements}\\
L^2(I',r^2\sin),&\text{partial measurements}
\end{cases}
\end{align}
where $Y$ are weighted $L^2$-data spaces.
Within the Hilbert space framework established by the well-posedness results in Sections \ref{sec:wellposed}–\ref{sec:wellposed-separate}, that is
\begin{align}\label{separate-eq}
\gamma\in \R^+, \quad \Omega\in X:= H^1(\Mc,m=0), \quad \Psi\in V_m:=\Hdm{2},\quad y\in Y,
\end{align}
we introduce the parameter-to-state map
\begin{align}\label{S}
\begin{aligned}
&S:\mathcal{D}(S) \to V_m,\quad \Dc(S):=\{(\gamma,\Omega)\in\R^+\cap X:   \eqref{non-separated-omega} \vee \eqref{smallness}\}\\
& S(\gamma,\Omega):=\Psi,\quad\text{where $\Psi$ is the weak solution to \eqref{eq:separated}.}
\end{aligned}
\end{align}
Together with the measurement operator $L$ in \eqref{Llinear},
the forward operator $F$ for the parameter identification is formulated as
\begin{align}\label{Flinear}
F:\Dc(S)\to Y \qquad F:=L\circ S.
\end{align} 
Note that $X$ is the space of real-valued functions, while $V_m$ and $Y$ are spaces of complex-valued functions; together, they form Gelfand triples. Compactness of the embedding $V_m \hookrightarrow Y$ yields compactness of the measurement operator $L$, thus of the forward operator $F$. In such cases, the inverse parameter problem becomes ill-posed, necessitating regularization to ensure stable reconstruction of the target coefficients. 

\subsection{Sensitivity and adjoints}\label{sec:adjoint}
Given the inherent nonlinearity of the inverse problem, iterative regularization methods such as Landweber-type and Newton-type algorithms are typically employed (see Section \ref{sec:reg_converge}).
For noisy data $y^\delta \in Y$, the Landweber iteration is given as
\[p_{k+1}=p_k-F'[p_k]^*(F(p_k)-y^\delta) \qquad k\leq K(y^\delta,\delta),\]
for the unknown parameter $p := (\gamma, \Omega)$. The stopping index $K(y^\delta, \delta)$ serves as the \emph{regularization parameter} and is determined according to the \emph{discrepancy principle}; see Section \ref{sec:numerics} for further details. This formulation highlights the main components of  gradient-based regularization schemes: sensitivity analysis and adjoint derivation; these are the main focus of this section.

For clarity, we will explicitly denote the dependence of $\Bf_m$ on the unknown parameters $p=(\gamma,\Omega)\in \mathcal{D}(S)$, writing
$
\Bbone{p}: V_m\to V_m^*.
$
As the mapping $p\mapsto \Bbone{p}$ is affine linear, the derivative is independent of $p$; we thus suppress this argument in the derivative and write
\[
\Bb':\mathbb{R}\times X\to \mathcal{L}(V_m, V_m^*), \qquad  
\Bbptwo{\delta\gamma,\delta\Omega}{\Psi}:= 
(\delta\gamma) \Lapm^2\Psi + im\alpha_{\delta\Omega} \Psi - im (\delta\Omega)\hl{\Lapm \Psi}.
\]
More precisely, $\Bb'(\delta\gamma,\delta\Omega)$ is the operator 
induced by the sesquilinear form $\mathbf{A}'_m(\delta\gamma,\delta\Omega):
V_m\times V_m\to \mathbb{C}$ defined by 
\[
\mathbf{A}'_m(\delta\gamma,\delta\Omega)[\Psi,\psi]
:=(\delta\gamma)\lsp \Lapm\Psi,\Lapm\psi \rsp
+im\alpha_{\delta\Omega}\langle\Psi,\psi\rangle - im\delta\Omega\langle\hl{\Lapm\Psi},\psi\rangle.
\]
We wish to emphasize that a similar analysis can be carried out for the unseparated equation with obvious modifications.

\begin{lemma}[Sensitivity] \label{lem:sensitivity}
The parameter-to-state map $S$ in \eqref{S} is Fr\'echet differentiable, and for any 
$p\in\mathcal{D}(S)$ and $\delta p\in \R \times X$ we have
\begin{align}\label{sen-eq}
S'[p]\delta p = -\Bbone{p}^{-1}\Bbp{\delta p}\Psi \in V_m,
\end{align}
where $\Psi:=\Bbone{p}^{-1}f$ is the weak solution to the primal equation \eqref{eq:separated}.
\end{lemma}
\begin{proof}
The result follows from the differentiability statement in the implicit function theorem 
(see, e.g., \cite[\S 4.7]{zeidler1986}) applied to the operator $G:\mathcal{D}(S)\times V_m\to V_m^*$, 
$G(p,\Psi) := \Bbtwo{p}{\Psi}-f$. Here, one
interprets $S$ as the implicitly defined function $G(p,S(p))=0$ and employs linearity of $G$ in the second argument as well as bounded invertibility of $\Bbone{p}:V_m\to V_m^*$.
\end{proof}

Differentiability of $F$ is straightforward from that  of $S$  and boundedness of the linear measurement operator $L$. We now derive the adjoint for different observations.

\begin{proposition}[Adjoint -- full observation]\label{prop:Flinear-adjoint}
Let $L$ be the full measurement operator in \eqref{Llinear}. Denote by 
$I_X:X^*\to X$ the Riesz isomorphism between dual spaces.\\
The Hilbert space adjoint of the derivative of $F$ in \eqref{Flinear} is given by
\begin{align}\label{FLinear-adjoint}
F'[p]^*: Y\to \R\times X \qquad F'[p]^*y= -\left[\begin{smallmatrix}1 &0\\0&I_X\end{smallmatrix}\right]
\,\re{[\Bbp{\cdot}\Psi]^\star
(\Bbone{p}^{-1})^\star y}
\end{align}
where $\Psi:=\Bbone{p}^{-1}f$ is again the solution to the separated equation \eqref{eq:separated} and 
\begin{align}\label{Banach-adj}
[\Bbp{\cdot}\Psi]^\star: V_m\to \R\times X^*, \quad (\Bbone{p}^{-1})^\star: V_m^*\to V_m
\end{align}
denotes the Banach space adjoints of $\Bbp{\cdot}\Psi:\R\times X\to V_m^*$ and 
$\Bbone{p}^{-1}: V_m^*\to V_m$.
\end{proposition}

\begin{proof}
Since the parameter space is real, we impose on $Y$ the real-valued inner product, the real part of the canonical complex-valued inner product, as
\begin{align*}
(f,g)_Y:=(\re{f},\re{g})_Y+(\im{f},\im{g})_Y=\re{(f,g)_Y^\C}.
\end{align*}
The 
linearization \eqref{FLinear-adjoint} and Banach space adjoints \eqref{Banach-adj} enable us to deduce
\begin{align*}
&\Big(F'[p]\delta p\,,\,y\Big)_Y^\C  = \lan S'[p]\delta p\,,y\ran_{L^2}^\C = -\lan\Bbone{p}^{-1}\Bbp{\delta p}{\Psi}\,,y\ran_{L^2}^\C \\
&= -\lan \Bbptwo{\delta p}{\Psi}\,,\, (\Bbone{p}^{-1})^\star y\ran_{V_m^*,V_m}^\C  = -\lan \delta p\,,\, [\Bbptwo{\cdot}{\Psi}]^\star(\Bbone{p}^{-1})^\star y\ran_{\R\times X, \R\times X^*}^\C. \nonumber
\end{align*}
Then with the real-valued inner product and the Riesz isomorphism $I_X$, we arrive at  
\begin{align*}
\Big(F'[p]\delta p\,,\,y\Big)_Y
&= \lan \delta p\,,\, -\re{[\Bbptwo{\cdot}{\Psi}]^\star(\Bbone{p}^{-1})^\star y} \ran_{\R\times X, \R\times X^*}\\ 
&= \Big( \delta p\,,\, -\left[\begin{smallmatrix}1 &0\\0&I_X\end{smallmatrix}\right]\re{[\Bbptwo{\cdot}{\Psi}]^\star(\Bbone{p}^{-1})^\star y}\Big)_{\R\times X}, 
\end{align*}
in agreement with the expression \eqref{FLinear-adjoint}. 
\end{proof}

We next derive the explicit expression for the Hilbert space adjoint \eqref{FLinear-adjoint}.
\begin{corollary}[Explicit form]\label{cor:adj_explicit}
The adjoint for full observation in Proposition \ref{prop:Flinear-adjoint} takes the explicit form
\begin{equation}\label{FLinear-adjoint_explicit}
\begin{split}
&F'[\gamma,\Omega]^*y=\begin{bmatrix} \re{\int_0^\pi(\Lapm^2\Psi) \overline{z}r^2\sin(\cdot)\,ds}\\[1ex]
 mI_X \im{\dfrac{\sin(\cdot)}{r^2}\dfrac{d}{d\theta}\(\dfrac{1}{\sin(\cdot)}\dfrac{d (\overline{\Psi} z )}{d\theta} \) - \(\Lapm\overline{\Psi}\)z}
\end{bmatrix}
\end{split}
\end{equation}
with the adjoint state $z:=(\Bbone{p}^{-1})^\star y$. 
\end{corollary}

\begin{proof}
To compute $[\Bbptwo{\cdot}{\Psi}]^\star$, take $(\delta \gamma,\delta\Omega)\in \R\times X$ and $z\in V_m$ and note that
\begin{align*}
\lan \Bbp{\delta\gamma,\delta\Omega}{\Psi}, z \ran 
&= \lan \delta\gamma \Lapm^2\Psi -im(\delta\Omega)\Lapm\Psi +im\alpha_{\delta\Omega}\Psi\,,\,z\ran\\
&=  
\delta\gamma\int_0^\pi (\Lapm^2\Psi)  \overline{z}r^2\sin(\cdot)\,d\theta
+ \lan \delta\Omega\cc  im(\Lapm\overline{\Psi})z\ran -\lan \alpha_{\delta\Omega} \cc im
\overline{\Psi} z)
\ran 
\end{align*}
To characterize the quantity involving $\alpha_{\delta\Omega}$, we may assume w.l.o.g.\ that $m\neq 0$ 
since $\alpha$ is irrelevant for $m=0$. Using the fact the 
$\Psi$ and $z$ satisfy the boundary conditions $\Gamma_m(\Psi)=\Gamma_m(z)=0$ 
under our regularity assumption on $\Omega$ according to Proposition 
\ref{prop:regularity}, we can carry out the following partial integrations:
\begin{align*}
&\lan \alpha_{\delta\Omega} \cc \overline{\Psi} z\ran\\
&\quad=\int_0^\pi \frac{d}{d\theta}\(\frac{1}{\sin\theta}\frac{d}{d\theta}\(\delta\Omega\sin^2\theta\)\)\Psi \overline{z} \,d\theta 
=\int_0^\pi \frac{-1}{\sin\theta}\frac{d}{d\theta}\(\delta\Omega\sin^2\theta\)\frac{d}{d\theta}\(\Psi \overline{z} \)\,d\theta
\\
&\quad = \int_0^\pi (\delta\Omega)(\theta)\sin^2\theta\frac{d}{d\theta}\(\frac{1}{\sin\theta}\frac{d}{d\theta}\(\Psi \overline{z} \)\)\,d\theta 
= \lan \delta\Omega\cc  \frac{\sin}{r^2}\frac{d}{d\theta}\(\frac{1}{\sin}\frac{d}{d\theta}\(\overline{\Psi} z \)\)\ran\nonumber
\end{align*}
Inserting these formulas into \eqref{FLinear-adjoint} yields \eqref{FLinear-adjoint_explicit}.
\end{proof}

\begin{remark}[Adjoint for partial and real measurements] \label{rem:adj-partial}
Note that the partial observation operator is the restriction operator $R\Psi:=\Psi|_{I'}$. Its Hilbert space adjoint is the extension-by-zero operator 
\[
R^*:L^2(I',r^2\sin)\to L^2(I,r^2\sin)\qquad 
R^*\Psi)(\theta):=\begin{cases}\Psi(\theta),&\theta\in I',\\
0,&\theta\in I\setminus I'.
\end{cases}
\]
The corresponding forward operator relates to that of full measurement via $F_{\mathrm{part}}[p] = RF_{\mathrm{full}}[p]$. Thus, 
$F_{\mathrm{part}}'[p]^* = F_{\mathrm{full}}'[p]^*R^*$ 
with $F_{\mathrm{full}}'[p]^*$ as in Corollary \ref{cor:adj_explicit}.

Similarly, if only the real part of the data can be observed, analogous formulas 
hold true with $R$ replaced by the real-part operator 
$\Re:L^2_{\C}(I,r^2\sin)\to L^2_{\R}(I,r^2\sin)$, of which the Hilbert space adjoint is the canonical embedding 
of $L^2_{\R}(I,r^2\sin)$ into $L^2_{\C}(I,r^2\sin)$. 
\end{remark}

\section{Tangential cone condition}\label{sec:reg_converge}

As discussed in Section \ref{sec:ip}, ill-posedness of our inverse problem necessitates regularization strategies for stable reconstruction. Gradient-based regularization methods, including Landweber, Newton-type schemes, and their variants, rely on three core elements: well-posedness of the forward operator $F$ (Section \ref{sec:well-posed}), the adjoint of the linearized operator (Section \ref{sec:adjoint}), and convergence guarantees.  The third requirement, which forms the primary focus of this section,  is ensured by structural assumptions on $F$ -- particularly, condition on nonlinearity and uniform boundedness of its derivative.

We address this through the tangential cone condition (TCC), a celebrated criterion first introduced in \cite{HNS95} that ensures local convergence of iterative regularization algorithms. Intuitively, if the forward map $F$ is excessively nonlinear, there is no general guarantee that gradient descent steps will remain within the vicinity of the true solution. 
The tangential cone condition offers a quantitative tool for assessing nonlinearity of $F$, particularly of compact operators, thereby enabling application of regularization methods to realistic problems. 
Notably, this condition also implies: characterization of the solution set via null space of the linearized operator, uniqueness of minimum-norm solutions,  
and local convergence of reconstruction sequences towards 
minimum-norm solutions \cite{KalNeuSch08}.

For a general model $F: X\ni p \mapsto y\in Y$, the  TCC states that the following inequality holds:
\begin{align*}
\|F(p)-F(\tilde{p})-F'[p](p-\tilde{p})\|_Y\leq C_{tc}\|p-\tilde{p}\|_X\|F(p)-F(\tilde{p})\|_Y \,\,\, \text{for all } p,\tilde{p}\in B^X_R(p^\dagger)
\end{align*}
where  $C_{tc}>0$ is the TCC constant, and  $B^X_R(p^\dagger)$ is the ball of radius $R$ around a ground truth $p^{\dagger}\in \mathcal{D}(F)$. In our setting, the TCC reads as
\begin{align}\label{tcc}
&\|F(\gamma,\Omega)-F(\tilde{\gamma},\tilde{\Omega})-F'[\gamma,\Omega]((\gamma,\Omega)-(\tilde{\gamma},\tilde{\Omega}))\|_Y \\
&\hspace{4cm}\leq C_{tc}\|(\gamma,\Omega)-\tilde{\gamma},\tilde{\Omega})\|_{\R\times X}\|F(\gamma,\Omega)-F(\tilde{\gamma},\tilde{\Omega})\|_Y \nonumber\\
&\hspace{8cm} \text{for all }  (\gamma,\Omega), (\tilde{\gamma},\tilde{\Omega})\in \ball. \nonumber
\end{align}
Weak variants of the TCC are discussed in \cite{Kindermann17}. Beyond the TCC, other structural conditions can also ensure convergence guarantees: range-invariance \cite{Kaltenbacher23-rangeinv}, convexity, and the Polyak–Łojasiewicz condition \cite{Nguyen24}

Although the TCC is a powerful tool, verification in practice can be challenging, even for specific examples. For a general verification strategy with full data, we refer the reader to \cite{TCC21}. Applications of TCC verification span a wide range: elliptic inverse problems \cite{HoffmanWaldNguyen:2021}, elastography \cite{Nakamura:MRE21, HubmerScherzer:TCC18}, electrical impedance tomography \cite{Kindermann21}, full waveform inversion \cite{EllerRolandRieder24}, and neural network-based inverse problems \cite{ScherzerHofmannNashed, AarsetHollerNguyen23}.

Our contribution is the TCC for inertial wave inversion under full measurement by introducing a lifted regularity strategy (Section \ref{sec:regularity}) that handles realistic $L^2$-data, and extending the verification to the restricted measurement regime, which is novel in this context. \hl{To this end, we employ the lifted regularity result prepared in Section \ref{sec:regularity}, which means that for this section, we assume
\[ \Omega\in H^2(\Mc;m=0).\]}
As a consequence of TCC, we obtain local unique identifiability of the unknown parameters, a property well-known to be important in ill-posed inverse problems.

\subsection{Full measurement}\label{sec:tcc-restricted}

\begin{lemma}[Uniform boundedness]\label{lem:uniform-N}
Given the setting in Proposition \ref{prop:regularity}. The operators $\Bbone{\gamma^{\dagger},\Omega^{\dagger}}^{-1}$ are bounded 
from $\Hdm{-4}$ to $L^2_{\diamond}(\Mc;m)$ for all 
$(\gamma^{\dagger},\Omega^{\dagger})\in \mathcal{D}(S)\cap (\R^+\times H^2(\Mc;m=0))$.  
Moreover, they are locally uniformly bounded in the sense that there exists
$\overline{R}(\gamma^{\dagger},\Omega^{\dagger})>0$ such that 
\[
N^{\gamma^\dagger,\Omega^\dagger}(R) := 
\sup\left\{ \|\Bbone{\gamma,\Omega}^{-1}\|_{H^{-4}\to L^2} : 
(\gamma,\Omega)\in\balltil\right\}
\] 
is finite for all $R\leq \overline{R}(\gamma^{\dagger},\Omega^{\dagger})$, 
and $\balltil\subset \mathcal{D}(S)$.  
\end{lemma}

\begin{proof}
Boundedness follows from the lifted regularity result in Proposition \ref{prop:regularity} and that $\Bbone{\gamma^{\dagger},\Omega^{\dagger}}: L^2_{\diamond}(\Mc;m)\to \Hdm{-4}$  
is the Banach adjoint of $\Bbone{\gamma^{\dagger},\Omega^{\dagger}}^\star: \Hdm{4}\to L^2_{\diamond}(\Mc;m)$. The local uniformity of norm bounds can be seen from the fact that the norms of 
$\gamma$, $\Omega$ appear in an affine liner manner in all upper bounds in the proof 
of Proposition \ref{prop:regularity}. 
\end{proof}

\begin{theorem}[TCC -- full measurement]\label{theo:tcc-full}
Given the setting in Proposition \ref{prop:regularity}. For $m\neq 0$, the forward operator $F$ with full measurements satisfies the 
TCC \eqref{tcc}. 
\end{theorem}

\begin{proof}
Recall $F=S$ for the full measurements . For any $p^{\dagger}=(\gamma^{\dagger},\Omega^{\dagger})\in \mathcal{D}(S)$, 
we have to show that there exists $R,C_\text{tc}>0$ such that 
\begin{align*}
\begin{split}
\|S(p+h)-S(p)-S'[p]h\|_{L^2}\leq C_\text{tc}
\|h\|_{\R\times H^2}\|S(p)-S(p+h)\|_{L^2}\\
\end{split}
\end{align*}
with $h:=\tilde{p}-p$, for all $p,\tilde{p}\in \ball$. 
Using the identity $\Bbone{p+h}-\Bbone{p}=\Bbp{h}$ and twice 
the identity $T^{-1}-S^{-1}=T^{-1}(S-T)S^{-1}$ 
for invertible operators $S$ and $T$, we obtain
\begin{align}\label{tcc-Q}
S(p+h)-S(p)-S'[p]h &= \Bbone{p+h}^{-1}f - \Bbone{p}^{-1}f + \Bbone{p}^{-1} \Bbp{h} S(p) \nonumber\\
&=[-\Bbone{p+h}^{-1}+\Bbone{p}]^{-1}] \Bbp{h}S(p) \nonumber\\
&=\Bbone{p}^{-1}\Bbp{h}\Bbone{p+h}^{-1}\Bbp{h} \Bbone{p}^{-1}f \nonumber\\
&=-\Bbone{p}^{-1}\Bbp{h}(S(p+h)-S(p)).
\end{align}
Therefore, it suffices to show that 
$\|\Bbone{p}^{-1}\Bbp{h}\|_{L^2\to L^2}\leq C_\mathrm{tc}\|h\|_{\R\times H^2}$.

We first estimate $\|\Bbp{h}\|_{L^2\to H^{-4}}$. Let $h=:(\delta \gamma,\delta \Omega)$, for $w\in L^2(\Mc;m)$  and $\varphi\in \Hdm{4}$, we obtain the following bound on the inner product:
\begin{eqnarray*}
\lefteqn{\lan \Bbp{h} w, \varphi\ran
=\lan \delta\gamma w, \Lapm^2\varphi\ran  
-\lan im(\delta\Omega) w, \Lapm\varphi\ran 
+ \lan imw, \alpha_{\delta\Omega}\varphi\ran }\\
&\leq&  |\delta\gamma|\|w\|_{L^2}\|\varphi\|_{H^4} + |m|\|\delta\Omega\|_{L^\infty}\|w\|_{L^2}\|\varphi\|_{H^2} 
+  C_{\alpha}|m|\|\delta\Omega\|_{H^2}\|w\|_{L^2}\|\varphi\|_{H^2} \\
&\leq&  C_\text{emb}\|h\|_{\R\times\Xtil}\|w\|_{L^2} \|\varphi\|_{H^4}
\end{eqnarray*}
where $C_{\mathrm{emb}}:=\sqrt{2}\max\big\{1,|m|C_{H^4\to H^2}
(C_{H^2\to L^\infty}+C_{\alpha})\big\}$ with $C_\alpha$ as in \eqref{eq:alpha_bound}.  
This shows that \[\|\Bbp{h}\|_{L^2\to H^{-4}}\leq C_\text{emb}\|h\|_{\R\times\Xtil}.\]
Then, together with the locally uniform boundedness in Lemma \ref{lem:uniform-N}, we achieve
\begin{equation}\label{tcc-B}
\begin{split}
\|\Bbone{p}^{-1}\Bbp{h}\|_{L^2\to L^2}
&\leq \|\Bbone{p}^{-1}\|_{H^{-4}\to L^2}\|\Bbp{h}\|_{L^2\to H^{-4}}\\
&\leq N^{\gamma^\dagger,\Omega^\dagger}(R)C_\text{emb}\|h\|_{\R\times H^2}.
\end{split}
\end{equation}
In this step, we observe the importance of the lifted regularity laid out in Section \ref{sec:regularity}, which is the key to Lemma \ref{lem:uniform-N}. This proves the TCC \eqref{tcc} in $B^X_R(p^\dagger)$.
\end{proof}

With the TCC, we further obtain unique identifiability of $\Omega$ given $\gamma$ and vice versa. We denote by $\mathrm{Arg}(z)\in[0,2\pi)$ the principal value of the argument of a complex number $z$.
\begin{corollary}[Unique identifiability]\label{coro:uniqueness}
Let $m\neq0$ and assume full measurement.
\begin{enumerate}
\item[(i)] Assume $\gamma$ is known. Suppose that \hl{for some 
subinterval $I'\subset (0,\pi)$ (possibly $I'=(0,\pi)$) we have}
\begin{equation}\label{nullset}
\Lapm\Psi\neq0 \text{ a.e.~on } \hl{I'},\quad \mathrm{Arg}(\Psi)-\mathrm{Arg}(\Lapm\Psi)\notin\{-\pi,0,\pi\} \text{ a.e.~on } \hl{I'}.
\end{equation}
Then $\Omega\hl{|_{I'}}$ is uniquely determined locally in $B_R^X(\Omega^\dagger)$. \hl{(We refer to Appendix 
\ref{appendix:baire} for a discussion of assumption \eqref{nullset}.)}
\item[(ii)]
i) Assume $\Omega$ is known. Suppose that 
\begin{equation}\label{nullset-gamma}
\Lapm^2 \Psi\neq 0 \text{ on some non-null subset of } I.
\end{equation}
Then $\gamma$ is uniquely determined by full measurement, locally in $B_R^X(\Omega^\dagger)$.
\end{enumerate}
\end{corollary}
\begin{proof}
(i) We first prove  that \hl{$\delta\Omega\in \Nc(F'[\Odag])$ implies 
$\delta\Omega|_{I'}=0$,}
 where $F=L\circ S$ with the embedding $L$. By injectivity of the embedding, it suffices to show  \hl{$\delta\Omega\in \Nc(S'[\Omega^\dagger])$ implies $\delta\Omega|_{I'}=0$.}
 Recall from Lemma \ref{lem:sensitivity} the linearized state equation 
\[
S'[\Odag]\delta\Omega= -\Bbone{\Odag}^{-1}(- im (\delta\Omega)\Lapm\Psi^\dagger+ im\alpha_{\delta\Omega} \Psi^\dagger),
\]
where $\Psi^\dagger=S(\Odag)$. Due to invertibility of $\Bbone{\Odag}$, the linearized state $S'[\Omega^\dagger]\delta\Omega=0\in L^2(\Omega)$ if and only if $\alpha_{\delta\Omega} \Psi^\dagger = (\delta\Omega)\Lapm\Psi^\dagger$ a.e.~on $I$. 

\hl{First assume that} $\alpha_{\delta\Omega} \Psi^\dagger=(\delta\Omega)\Lapm\Psi^\dagger=0$ a.e.~on 
\hl{$I'$. This implies }
$\delta\Omega=0$ a.e~on \hl{$I'$} by the first assumption in \eqref{nullset}.

If this is not the case, 
there exists a non-null subset of \hl{$I'$} where $\alpha_{\delta\Omega} \Psi^\dagger = (\delta\Omega)\Lapm\Psi^\dagger\neq 0$ a.e. This allows  taking $\mathrm{Arg}$ here; noting that $\delta\Omega$, $\alpha_{\delta\Omega}$ are real functions, it necessarily holds that $\mathrm{Arg}(\Psi)-\mathrm{Arg}(\Lapm\Psi)\in\{-\pi,0,\pi\}$ a.e.~on this non-null subset, contradicting the second part of assumption \eqref{nullset}. Thus, \hl{the second case is impossible, and we always have $\delta\Omega=0$ a.e~on $I'$.}

From this, along with the characterization of the solution set given the TCC in \cite[Proposition 2.1]{HNS95}, we deduce $\Omega^*-\Odag\in\Nc(F'[\Odag])$  \hl{and thus $\Omega^*|_{I'}=\Odag|_{I'}$}
for any further solution $\Omega^*\in B_R^X(\Omega^\dagger)$ of the inverse problem. 

(ii) Assumption \eqref{nullset-gamma} implies that $F'[\gamma^\dagger]\delta\gamma=\delta\gamma\Lapm^2\Psi^\dagger=0\in L^2(\Omega)$ holds if and only if $\delta\gamma=0\in\R$. The remainder of the argument proceeds as above.
\end{proof}

\subsection{Partial measurement including Cauchy data}\label{sec:tcc-restricted}
In the last part of the analysis, we consider the restricted measurement scenario where observational latitude range is limited to $I' = (\epsilon, \pi-\epsilon)$ as described in Section \ref{sec:data}. For simplification, we additionally assume that 
Cauchy data at the boundary of $I'$ can be measured exactly. 
We follow the analysis developed for the full observation, but specifically leveraging the well-posedness and regularity results for the restricted domain from Proposition \ref{prop:wellposed-partial} in  Section \ref{sec:wellposed-restricted}.

\begin{theorem}[TCC -- partial measurement]\label{theo:tcc-restrict}
Consider the setting in Proposition \ref{prop:wellposed-partial}, \ref{prop:wellposed-partial-liftedreg}. If one imposes the constraint
\begin{align}\label{tcc-partial-extra-info}
\Psi|_{\partial I'}=\Psi^\dagger|_{\partial I'}\quad\text{and}\quad \Psi'|_{\partial I'}=(\Psi^\dagger)'|_{\partial I'},
\end{align}
where $\Psi^\dagger$ is the exact state, then for $m\neq 0$, the forward operator $F$ with partial measurements satisfies the TCC \eqref{tcc}. 
\end{theorem}
\begin{proof}
The proof is almost identical to that of  Theorem \ref{theo:tcc-full}.
Note that \hl{if one proceeds as in \eqref{tcc-Q}}, $S(p+h)-S(p)$ and $S'(p)$ both satisfy homogeneous Cauchy boundary conditions on $\partial I'$. \hl{The fact that all the invertible operators act on data with homogeneous Cauchy boundary condition, as proven in Section \ref{sec:wellposed-restricted},} allows us to carry out partial integration as before and prove that
\begin{align*}
\|\Bbone{p}^{-1}\Bbp{h}\|_{L^2(I')\to L^2(I')}
&\leq \|\Bbone{p}^{-1}\|_{H^{-4}(I')\to L^2(I')}\|\Bbp{h}\|_{L^2(I')\to H^{-4}(I')}\\
&\leq N^{\gamma^\dagger,\Omega^\dagger}(R)C_\text{emb}\|h\|_{\R\times H^2(I';m=0)}\\
&\leq N^{\gamma^\dagger,\Omega^\dagger}(R)C_\text{emb}\|h\|_{\R\times H^2(\Mc;m=0)}
\end{align*}
for any $p$, $p+h$ in the ball $B^X_R(p^\dagger)$.
\end{proof}

The boundary data assumption in \eqref{tcc-partial-extra-info}, requiring knowledge of the exact state at $\partial I'$, is important for the proof of Theorem \ref{theo:tcc-restrict}. Relaxing this condition presents an interesting question for future research. With this theoretical discussion concluded, we now turn to the numerical experiments.

\section{Numerical experiments}\label{sec:numerics}

This section is dedicated to numerical experiments on synthetic data, using analytical ground truths.

\paragraph{\hl{Regularization method.}} We implemented the accelerated Nesterov-Landweber method to simultaneously reconstruct the scalar viscosity and latitudinal differential rotation parameters $(\gamma, \Omega)$. 
A convergence analysis of this method has been carried out in  \cite{HubmerRamlau17, Neubauer17} under the tangential cone condition, which we have been verified in Section \ref{sec:reg_converge}. \hl{The Nesterov-Landweber iteration is expressed in pseudocode in Algorithm \ref{algorithm}}. 

\color{blue}
\begin{algorithm}[h!]
\caption{Nestorov-Landweber with backtracking}\label{algorithm}
\begin{algorithmic}[1]
\Require \hl{Initial guess $p_{-1}=p_0:=(\gamma_\mathrm{int},\Omega_\mathrm{int})\in \mathbb{R}\times H^2(\Mc; m=0)$, $z_0:=0\in H^2(\Mc; m=0)$, starting index $k=0$, acceleration parameter $\alpha=3$, discrepancy constant $\tau>2$, noise level $\delta > 0$, noisy data $y^\delta\in Y$}
\While{\hl{discrepancy principle $\|F(z_{k}) - y^\delta\| \leq \tau\delta$ not satisfied}}
\State \hl{$z_k \gets p_k + \frac{k-1}{k+\alpha-1}\left(p_k-p_{k-1}\right)$ \hfill $\triangleright$ acceleration step} 
\State \hl{Choose step size $\mu_k>0$ via backtracking line search}
\State \hl{$p_{k+1} \gets z_k-\mu_k F'[z_k]^*\left( F(z_k)-y^\delta \right)$ \hfill $\triangleright$ parameter update step}
\State \hl{$k\gets k+1$}
\EndWhile
\State \hl{Set stopping index $K(y^\delta,\delta)\gets k$}\\
\Return \hl{$p_{K(y^\delta,\delta)}$}
\end{algorithmic}
\end{algorithm}
\color{black}

In \hl{Algorithm \ref{algorithm}}, the forward map $F$ corresponds to the wave equation at single frequency $\omega$ and longitudinal wavenumber $m$. We allow for different observation strategies: full data, leaked data and real part measurements. The Hilbert space adjoints $F'[z_k]^*$ are derived in Section \ref{sec:adjoint}. 

\paragraph{\hl{Discretization and data.}} The numerical wave solver uses a fourth-order finite difference scheme \cite{VersyptBraatz14} on a uniform grid of 100 points spanning the latitude range $(0, \pi)$. To generate ground truths while avoiding inverse crimes, we choose analytic $\Psi_m$,$\gamma$, and $\Omega$ with Neumann (Fig.~\ref{fig-full}) or  Dirichlet (Fig.~\ref{fig-leak-noise}) boundary, then explicitly obtain the corresponding source $f$. The discrete measurement $\underline{y}^\delta$ is corrupted by independent additive complex Gaussian noise with variance $\sigma^2$ on 100 equidistant measurement points and $\sigma$ 
chosen to have relative noise levels 
$$\|\underline{y}^{\delta}-\underline{y}\|_2/\|\underline{y}\|_2\in \{0.01, 0.05, 0.1, 0.2\}$$ 

\paragraph{\hl{Initialization and stopping rule.}} We initialize the algorithm with $\gamma_\mathrm{int} = -3\gamma_\mathrm{true}$, $\Omega_\mathrm{int} = 0$ corresponding to no prior knowledge of ground truths, and iterate with a step size $\mu_k$ determined by backtracking line search. \hl{Regarding the discrepancy principle employed in Algorithm \ref{algorithm}, the lower value of $\tau=1.03$ was found to yield better results in practice than the theoretically required $\tau>2$ \cite{HubmerRamlau17}}. The algorithm typically completes in 200 iterations within one second on an i7-1255U CPU (4.70 GHz)\footnote{Code is available on Github at \url{https://github.com/TramNguyenAca/Inertial_Waves}.}.

\begin{figure}[htb!]
\centering
\includegraphics[scale=0.8]{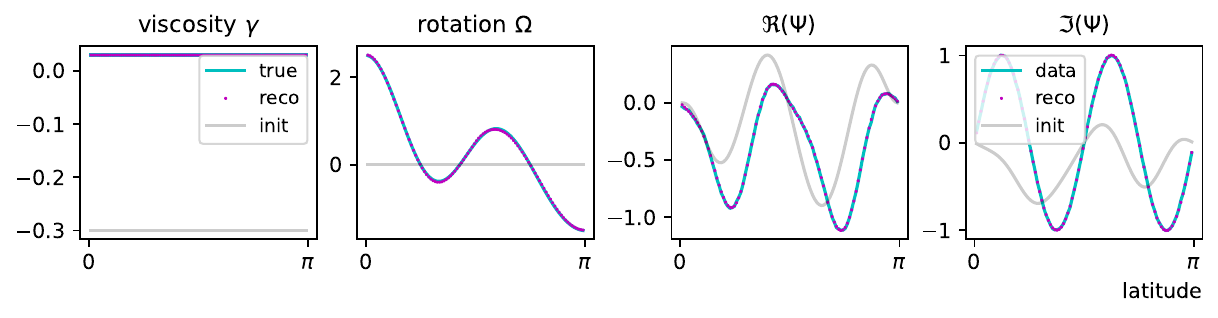}
\includegraphics[scale=0.8]{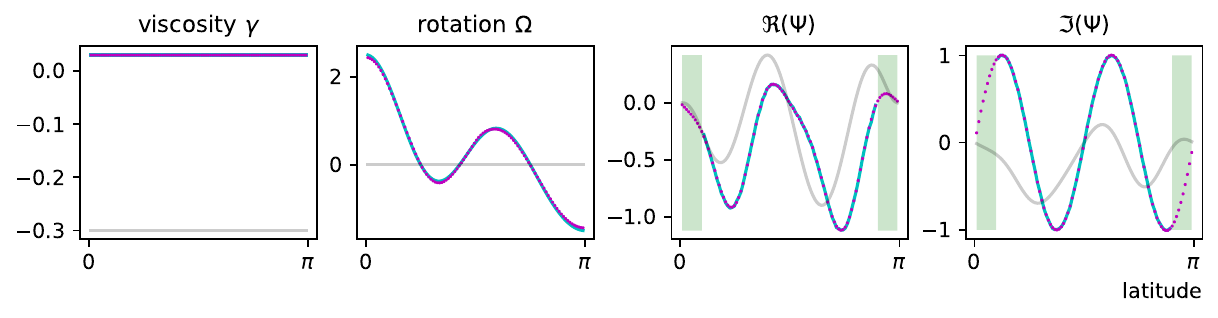}
\includegraphics[scale=0.8]{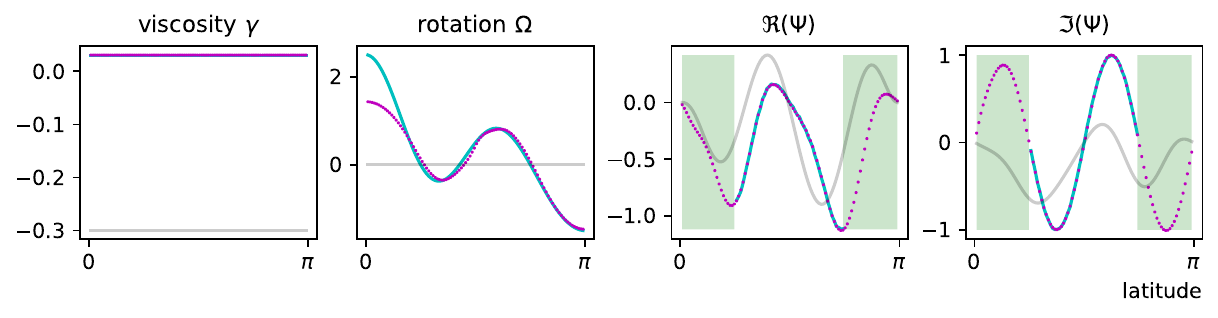}
\includegraphics[scale=0.8]{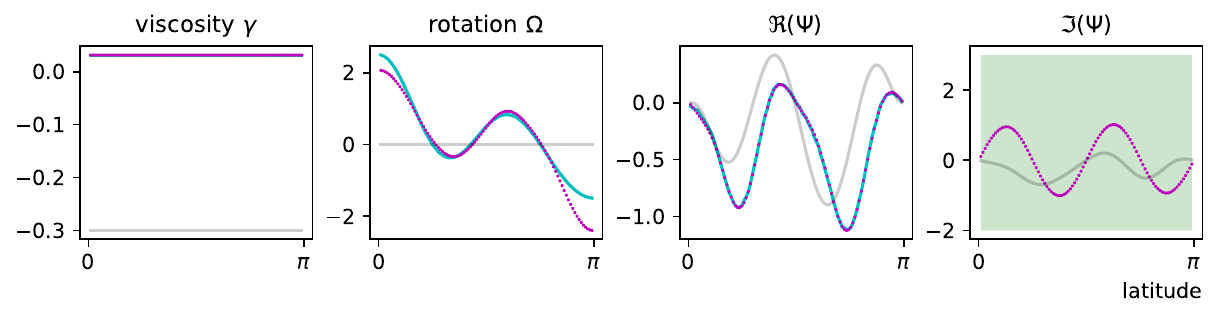}
\caption{1\% data noise. Top to bottom: full and leaked data (filled area) at different levels: 20\%, 50\%, missing imaginary part.}\label{fig-full}
\end{figure}

\begin{figure}[htb!]
\centering
\includegraphics[scale=0.8]{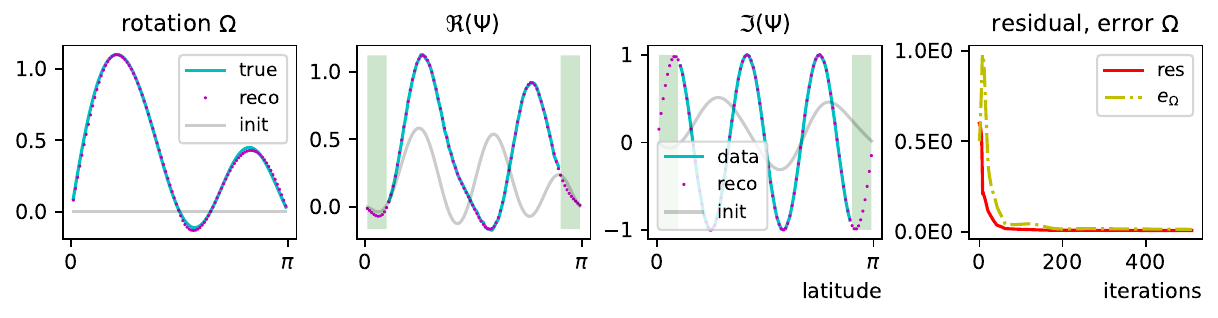}
\\[2ex]
\includegraphics[scale=0.8]{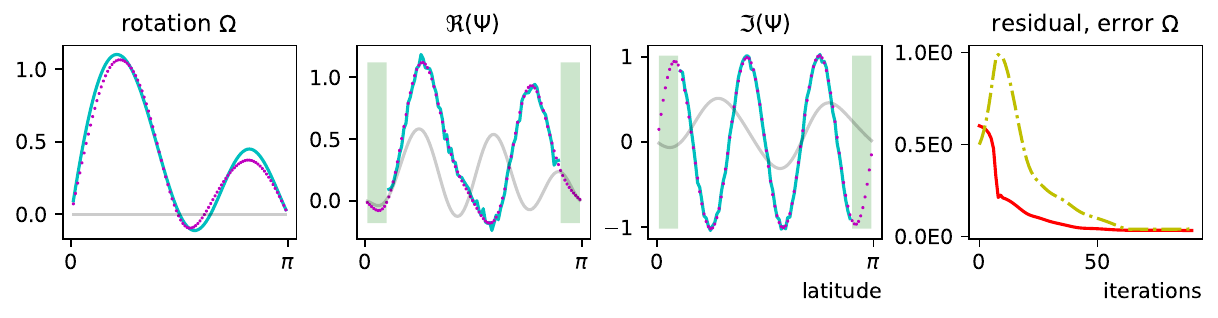}
\\[2ex]
\includegraphics[scale=0.8]{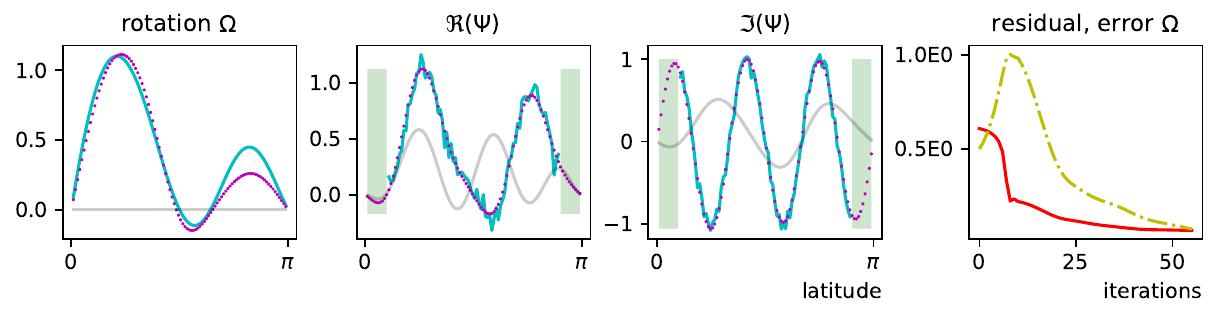}
\\[2ex]
\includegraphics[scale=0.8]{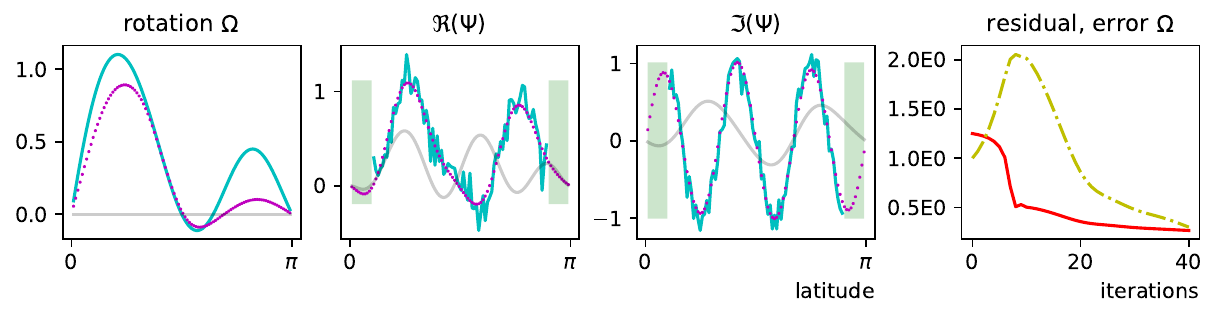}
\caption{Leaked data. Top to bottom: 1\%, 5\%, 10\%, 20\% relative noise level.}
\label{fig-leak-noise}
\end{figure}

\paragraph{Full data.} Figure \ref{fig-full}, row 1 shows the simultaneous recovery of $(\gamma, \Omega)$ from full measurement of the state $\Psi$ at $(\omega, m)=(3,3)$ and with $1\%$ data noise. Despite an initial guess far from the ground truth, the reconstruction closely approximates the exact parameters. The approximated wave field also matches the true state in both real and imaginary parts.

\paragraph*{Leaked data at poles.} Figure \ref{fig-full}, rows 2–3 present the practical case where no data are available near the poles. 
In our numerical experiments, we did not impose additional Cauchy boundary data in Theorem \ref{theo:tcc-restrict}, and simply computed the adjoint as described in Remark \ref{rem:adj-partial}. With 20\% of the data leaked (row 2), the reconstruction quality remains nearly as high as with full data. For 50\% data loss (row 3), the recovery of $\Omega$ moderately deteriorates. Nonetheless, the algorithm exhibits robust performance, consistent with our convergence results in Section \ref{sec:tcc-restricted}.

\paragraph*{Real part measurement.} Figure \ref{fig-full}, row 4 presents the case where only the real part of $\Psi$ can be measured. Compared to row 3 (which also has 50\% data leakage, but full complex data), the absence of the imaginary component further deteriorates the reconstruction.

\paragraph*{Noisy data.} Finally, Figure \ref{fig-leak-noise} reports the results for $(\omega, m)=(1,2)$ with 20 \% data leakage and different noise levels. 
At 1\% noise (row 1), the recovered parameters remain very close to the truth, 
as evidenced by small relative errors ($\||p-p_{\text{true}}\|_{L^2}/\|p_{\text{true}}\|_{L^2}$). 
At 5\% noise (row 2), reconstruction quality is slightly reduced, while at 20\% noise (row 4), the outcome degrades significantly.

\section{Conclusions and Outlook}\label{sec:outlook}

In this paper, we investigated the modeling, as well as the associated forward and inverse problems for inertial waves on the surface of the Sun.  
Formulating the dynamics in terms of a stream function leads to a fourth-order  
elliptic differential equation on the sphere and its separated version for each longitudinal-wavenumber. We rigorously proved 
well-posedness of these equations, including explicit smallness conditions 
for uniqueness, as well as conditions for non-uniqueness via analytic Fredholm theory, both in agreement with empirical findings in helioseismology.
Our studies represent the starting point for further theoretical and 
numerical studies of resonances and inertial modes, as well as the associated inverse problems. 

For the inverse problem of simultaneous identification of differential rotation and viscosity, we established convergence guarantees for the iterative reconstruction process via the tangential cone condition. Furthermore, we established local unique identifiability of rotation profile when viscosity is known, and vice versa. For the case of partial observations, the tangential cone condition required the additional assumption of exact observations of Cauchy boundary data. It is natural to ask whether this assumption can be relaxed, e.g., by including the Cauchy data in the range of the operator. We point out that at least for a finite number of $m$ and $\omega$, uniqueness of rotation at the unobserved latitudes seems unrealistic by dimensionality arguments. 

Future extensions of this work include incorporating realistic stochastic sources of solar excitation \cite{PhilidetGizon23} within a passive imaging framework \hl{\cite{Bjoern23,nguyen-passive}}.
Additionally, we will investigate the nonlinear equation for finite amplitude waves, which is required to interpret the time evolution of linearly unstable inertial modes \cite{Bekki2024}. 
Another promising avenue involves leveraging data-driven model discovery techniques \cite{AarsetHollerNguyen23, nguyen25} to refine existing theoretical inertial wave models by learning hidden laws. 

\appendix
\section{Differential geometry and function spaces on spheres}\label{sec:appendix}
Let $\mathbb{S}^2:=\{\rv\in\mathbb{R}^3:|\rv|=1\}$ denote the unit sphere in $\mathbb{R}^3$ 
and $\Mc = \{r \rv:\rv\in\mathbb{S}^2\} = \{\rv\in\mathbb{R}^3:|\rv|=r\}$ the sphere of 
radius $r>0$. 
In this appendix, we recall some differential operators on $r\mathbb{S}^2$ and their properties as well as function spaces on $\Mc$, which are used throughout this paper.  

\paragraph*{Spherical coordinates.}
We use $
\rv(r,\theta,\phi):= (
r\sin\theta\cos\phi,
r\sin\theta\sin\phi,
r\cos\theta
)
$ as definition of spherical coordinates with radial distance 
$r>0$, polar angle $\theta\in (0,\pi)$ and azimuthal angle $\phi\in (0,2\pi)$. In our notation, we do not distinguish between 
a function $f:\mathbb{R}^3\to \mathbb{R}$ and its counterpart $f\circ\rv:(0,\infty)\times(0,\pi)\times(0,2\pi)\to\mathbb{R}$ in spherical coordinates. Note that the vectors
$\{\base{r}(\rv), \base{\theta}(\rv), \base{\phi}(\rv)\}$ defined by $\base{r}(\rv):=\diffq{\rv}{r}/
|\diffq{\rv}{r}|$ and similarly for $\base{\theta}$ and $\base{\phi}$ form
a local orthonormal basis for all $\rv$. 

\paragraph*{Differential operators on the sphere.}
A vector-valued function $\fv:\Mc\to \mathbb{R}^3$ is called \emph{tangential} if 
$\fv(\rv)\cdot\base{r}(\rv)=0$ for all $\rv$. For a tangential vector field $\fv$ and 
a scalar $f$ on $\Mc$, consider any smooth extensions $\tilde{\fv}$ and $\tilde{f}$ to 
a neighborhood of $\Mc$ in $\mathbb{R}^3$. Then it is possible to define the surface (or 
horizontal) gradient 
$\Grad f$, the surface divergence $\Div \fv$, and the scalar and vectorial surface 
curl $\Curl \fv$ and $\vCurl f$ such that 
\begin{align}\label{eq:SurfaceDiffOp}
\begin{aligned}
&\Grad f  = \big(\grad\tilde{f} - \diffq{\tilde{f}}{r}\base{r}\big)|_{\Mc},
&&\Div\fv = \diver \tilde{\fv}|_{\Mc},\\
&\vCurl f = \curl (\tilde{f}\base{r})|_{\Mc},
&&\Curl \fv = (\curl \tilde{\fv}\cdot\base{r})|_{\Mc}
\end{aligned}
\end{align}
independent of the choices of $\tilde{f}$ and $\tilde{\fv}$ (see \cite[p.~72--73]{nedelec:01}). 
Here $\Div\fv$ and $\Curl\fv$ are scalar functions, while $\Grad f$ and $\vCurl v$ are tangential 
vector fields on $\Mc$. In spherical coordinates, writing 
$\fv= f_{\theta}\base{\theta}+f_{\phi}\base{\phi}$ with $f_\theta:=\fv\cdot\base{\theta}$ 
and $f_{\phi}:=\fv\cdot\base{\phi}$, we have
\begin{align}\label{eq:diffops_on_sphere}
\begin{aligned}
&\Grad f= \frac{1}{r} \diffq{f}{\theta} \base{\theta} + \frac{1}{r\sin\theta}\diffq{f}{\phi}\base{\phi},
&&\Div\fv = \frac{1}{r\sin\theta}\left(\diffq{}{\theta}(f_{\theta}\sin\theta)+\diffq{f_{\phi}}{\phi}\right),\\
&\vCurl f= \frac{1}{r\sin\theta}\diffq{f}{\phi}\base{\theta} -
\frac{1}{r} \diffq{f}{\theta} \base{\phi},
&&\Curl\fv = 
\frac{1}{r\sin\theta}\left(\diffq{}{\theta}(f_{\phi}\sin\theta)-\diffq{f_{\theta}}{\phi}\right).
\end{aligned}
\end{align}
These differential operators satisfy the identities (\cite[Thm.~2.5.19]{nedelec:01}):
\begin{equation}
\begin{split}
&\Curl\Grad = 0,\qquad \Div\vCurl =0,\\
&\int_{\Mc}\Grad f\cdot \fv \D s = -\int_{\Mc} f \Div\fv\D s,\,\,\,
  \int_{\Mc}\vCurl f\cdot \fv \D s = \int_{\Mc} f \Curl\fv\D s.
\end{split}
\end{equation} 
The Laplace-Beltrami operator on $\Mc$ is given by
\begin{align}
\label{eq:LaplaceBeltrami}
\Laph f=\Div\Grad f = -\Curl\vCurl f, 
\end{align}
and the vector Laplacian or Hodge operator acting on tangential vector fields $\fv$ is defined as 
\begin{align}\label{eq:vLapl}
\vLaph\fv = \Grad\Div\fv - \vCurl\Curl\fv.
\end{align}

We also need a decomposition of $L^2_{\mathrm{tan}}(\Mc):=\{\fv\in L^2(\Mc,\mathbb{R}^3:
\fv\cdot \base{r}\equiv 0\}$ into two orthogonal subspaces. 
Since $\Mc$ is simply connected, there exist for each tangential vector field $\fv\in L^2_{\mathrm{tan}}(\Mc)$ two functions 
$\varphi,\psi\in H^1(\Mc)$ such that 
\begin{align}\label{eq:HelmholtzDecomposition}
\fv=\Grad\varphi +\vCurl\psi,
\end{align}
and $\langle\Grad\varphi,\vCurl\psi\rangle_{L^2}=0$ (see \cite[Eq.~(5.6.24)]{nedelec:01}). 

\paragraph*{Function spaces.}
We define Sobolev spaces $H^s(\Mc)$ as Hilbert scale generated 
by the operator $(I-\Delta_{\rm h})^{1/2}$: 
\[
H^s(\Mc):=\mathrm{dom}((I-\Delta_{\rm h})^{s/2}) \qquad \mbox{with}\quad
\|f\|_{H^s}:=\|(I-\Delta_{\rm h})^{s/2}f\|_{L^2}
\] 
for $s\geq 0$. 
It is straightforward that these are Hilbert spaces, and that $H^0(\Mc):=L^2(\Mc)$.  
Moreover, we set $H^{-s}(\Mc):= H^s(\Mc)^*$ with the natural norm. We further introduce 
subspaces of functions with mean $0$ by
\begin{align}\label{eq:mean0space}
L^2_{\diamond}(\Mc) := \{f\in L^2(\Mc):\langle f,1\rangle = 0\},\quad 
\Hd{s}:= \{f\in H^s: \langle f, 1\rangle = 0\} 
\end{align}
for $s\in \mathbb{R}$. These spaces are also complete with the same norms, but we will later 
define a slightly more convenient equivalent norm on $\Hd{s}$. 

The following embeddings hold true (see 
\cite[Propositions 3.2-3.3]{Taylor}, \cite[Theorem 3.5]{Hebey}):
\begin{subequations}\label{eqs:embeddings_sphere}
\begin{align}
\label{eq:embedding_sobo}
& H^s(\Mc)\embed H^k(\Mc)&&\text{compact for  $s>k$}\\
\label{eq:embedding_Lp}
& H^s(\Mc)\embed L^p(\Mc)&&\text{continuous for $s\geq \max(0,1-\tfrac{2}{p})$ and $p\in [1,\infty)$}\\
\label{eq:embedding_Ck}
&H^s(\Mc)\embed C^k(\Mc)&& \text{compact for $s>k+1$  with $k\geq 0$}.
\end{align}
\end{subequations}
\paragraph*{Azimuthal Fourier transform.}
Writing a function $f\in L^2(\Mc)$ in spherical coordinates $f=f(\theta,\phi)$, we can perform 
a Fourier transform in $\phi$ to expand $f$ into a Fourier series
\begin{align}\label{eq:Fourier_phi}
f(\theta,\phi) = \frac{1}{\sqrt{2\pi}}\sum_{m=-\infty}^m \hat{f}_m(\theta)e^{im\phi}.
\end{align}
Let us define $e_m(\phi):=(2\pi)^{-1/2}e^{im\phi}$ and $(\hat{f}_m\otimes e_m)(\theta,\phi):=\hat{f}_m(\theta)e^{im\phi}$. Since $\hat{f}_m\otimes e_m \perp \hat{f}_n\otimes e_n$ in $L^2(\Mc)$ 
and in $L^2_{\diamond}(\Mc)$ 
for $m\neq n$, the expression \eqref{eq:Fourier_phi} induces orthogonal decompositions 
\begin{align}\label{eq:L2_decomp}
L^2(\Mc)=\bigoplus\nolimits_{m=-\infty}^\infty L^2(\Mc;m),\quad
L^2_{\diamond}(\Mc)=\bigoplus\nolimits_{m=-\infty}^\infty L^2_{\diamond}(\Mc;m)
\end{align}
with $\hat{f}_m\otimes e_m \in L^2(\Mc;m)$. As the spherical harmonics $\{Y_{lm}:l\in \mathbb{N}_0,m\in\mathbb{Z},|m|\leq l\}$ form a complete orthonormal basis 
of $L^2(\Mc)$ and $Y_{lm}\in L^2(\Mc;m)$, we have 
\begin{align*}
&L^2(\Mc;m)= L^2_{\diamond}(\Mc;m) = 
\overline{\mathrm{span}\{Y_{lm}:l\geq |m|\}}^{L^2},\qquad m\in\mathbb{Z}\setminus \{0\},\\
&L^2(\Mc;0) = \overline{\mathrm{span}\{Y_{l0}:l\in \mathbb{Z}\}}^{L^2},\quad 
L^2_{\diamond}(\Mc;0) = \overline{\mathrm{span}\{Y_{l0}:l\in \mathbb{Z}\setminus\{0\}\}}^{L^2},
\end{align*}
where $\mathrm{span}$ denotes the set of \emph{finite} linear combinations. 
The Laplace-Beltrami  
operator is diagonal with respect to the decompositions \eqref{eq:L2_decomp}. More precisely, 
it follows from \eqref{eq:SurfaceDiffOp} and \eqref{eq:LaplaceBeltrami} that
\begin{align}\label{eq:Lap_decomp}
&\Laph f = \sum_{m=-\infty}^\infty (\Lapmh \hat{f}_m)\otimes e_m
\quad \mbox{with}\quad 
\Lapmh:=\frac{1}{r^2\sin\theta}\frac{d}{d\theta}\(\sin\theta\frac {d }{d \theta }\)-\frac{m^2}{r^2\sin^2\theta}\,
\end{align}
where by abuse of notation, we denote $\Lapm (\hat{f}_m\otimes e_m):=(\Lapmh\hat{f}_m)\otimes e_m$, thus also
\begin{align}\label{eq:Grad_decomp}
\Gradm:= \(\frac{1}{r} \frac{d}{d\theta}; \frac{im}{r\sin\theta}\).
\end{align}
We obtain $H^s$-orthogonal decompositions 
\begin{align}\label{eq:Hs_decomp}
H^s(\Mc) = \bigoplus\nolimits_{m\in\mathbb{Z}} H^s(\Mc;m)\quad \mbox{and}\quad
\Hd{s} = \bigoplus\nolimits_{m\in\mathbb{Z}} H_{\diamond}^s(\Mc;m)
\end{align}
for any $s\in \mathbb{R}$, where $H^s(\Mc;m)$ and $H_{\diamond}^s(\Mc;m)$ are the closures of 
$\mathrm{span}\{Y_{lm}:l\geq |m|\}$ in $H^s(\Mc)$ and $\Hd{s}$, respectively. 
Again, $H^s(\Mc;m)=H_{\diamond}^s(\Mc;m)$ for $m\in \mathbb{Z}\setminus\{0\}$ 
and  $H^s(\Mc;0)=H_{\diamond}^s(\Mc;0)\oplus \mathrm{span}\{1\}$. 
As $-\Laph Y_{lm}(\cdot/r) = \frac{l(l+1)}{r^2}Y_{lm}(\cdot/r)$, 
and $Y_{00}(\cdot/r)\notin L^2_{\diamond}(\Mc)$, 
it follows that $\langle-\Laph f,f\rangle \geq 2/r^2 \|f\|_{L^2}^2$ for all $f\in L^2_{\diamond}(\Mc)$. Consequently, 
\[
\|f\|_{H_{\diamond}^s}:=\|(-\Laph)^{s/2} f\|_{L^2}, \qquad s\geq 0
\] 
defines a norm on $\Hd{s}$ which is equivalent to the $\|\cdot\|_{H^s}$-norm. 
We will 
consider this as the standard norm on $\Hd{s}$ as it has the convenient properties 
\begin{align}\label{eq:Hdnorms12}
\begin{aligned}
&\|f\|_{H_{\diamond}^2} = \|\Laph f\|_{L^2} \quad\mbox{and}\quad\\
&\|f\|_{H_{\diamond}^1}^2 = \langle -\Laph f,f\rangle 
= \langle -\Div\Grad f, f\rangle = \|\Grad f\|_{L^2}^2.
\end{aligned}
\end{align}
Defining $\Hd{-s}:=\Hd{s}'$ for $s>0$ \hl{and using the fact that the functional calculus is an algebra homomorphism, it follows that} 
\begin{align}\label{eq:Laph_isometric_iso}
\Laph: \Hd{t}\to \Hd{t-2}\quad \text{is an isometric isomorphism for all $t\in\mathbb{R}$.}
\end{align}
Defining $L^p(\Mc;m)$ as the
closure of $\mathrm{span}\{Y_{lm}:l\geq |m|\}$ in $L^p(\Mc)$, the embeddings \eqref{eqs:embeddings_sphere} implies the embeddings in the decomposed spaces
\begin{subequations}\label{eqs:separated_embeddings_sphere}
\begin{align}
& H^s(\Mc;m)\embed H^k(\Mc;m)&&\text{compact for  $s>k$}\\
&H^s(\Mc;m)\embed L^p(\Mc;m)&&\text{continuous for $s\geq \max(0,1-\tfrac{2}{p})$, $p\in [1,\infty)$.}
\end{align}
\end{subequations}

\section{Boundary conditions}\label{appendix:boundary}
The following lemma describes homogeneous boundary conditions satisfied by 
the $\theta$-dependent factors of separated smooth solutions: 
\begin{lemma}\label{lemm:derivatives_at_poles}
Let $m\in\Z$ and $k\in \N_0=\{0,1,2,\dots\}$. A function of the form $u(\theta,\phi) = u_m(\theta)\exp(im\phi)$ in spherical 
coordinates belongs to $C^k(\Mc)$ if and only if the following three conditions are satisfied:
\begin{enumerate}[(i)]
\item\label{it:um_smooth} $u_m\in C^k([0,\pi])$,
\item\label{it:oddeven} $u_m^{(j)}(\theta) =0$ for $\theta\in\{0,\pi\}$ and all $j\in\N_0$ 
with $j\leq k$ and $(-1)^{|m|+j}=-1$,
\item\label{it:firstderivatives} $u_m^{(j)}(\theta)=0$ for $\theta\in\{0,\pi\}$ and all 
$j\in\N_0$ with $j<m$.
\end{enumerate}
\end{lemma}

\begin{proof}
It is clear that $u_m\in C^k([0,\pi])$ is equivalent to 
$u\in C^k(\Mc\setminus\{P_+,P_-\})$ with the poles $P_{\pm}:=(0,0,\pm 1\}$. 
Moreover, as $u_m(\theta) = u(\theta,0)$, we have $u_m\in C^k([0,\pi])$ if $u\in C^k(\Mc)$.

To study regularity at the pole  $P_+$, 
consider the smooth map $M:D:=\{(x,y)\in\R^2:x^2+y^2\leq r^2/2\}\to \Mc$ of $\Mc$ 
defined by $M(x,y):= \sqrt{r^2 - x^2-y^2}$. By definition, 
$u$ is $k$-times differentiable at $P_+$ if and only if $u\circ M$ is $k$ times differentiable at $0$. And this is the case if and only if there exists a polynomial 
$p(x,y) = \sum_{j=0}^{k}\sum_{n=0}^j c_{j,n} x^{n} y^{j-n}$ of degree $\leq k$ such that 
\begin{align}\label{eq:polynomial_approximation}
\left|(u\circ M - p)(x,y)\right| = o\left(\sqrt{x^2+y^2}^{k}\right)\qquad \text{as }(x,y)\to 0.
\end{align} 
Writing $x=\frac{\rho}{2}(e^{i\phi}+e^{-i\phi})$ and $y=\frac{i\rho}{2}(e^{i\phi}-e^{-i\phi})$ 
in polar coordinates and using binomial expansions, we obtain
\[
\sum_{n=0}^j c_{j,n} x^{n} y^{j-n} = \rho^j \sum_{l=0}^j d_{j,l} e^{i(j-2l)\phi} 
\]
for some coefficients $d_{j,l}\in \C$. As
$(u\circ M)(\rho \cos\phi,\rho\sin\phi) = u_m(\rho)\exp(m\phi)$, 
the coefficients $d_{j,l}$ must vanish unless $j-2l=m$, so
\[
u_m(\rho) = \sum_{l\geq 0, m+2l\leq k} d_{m+2l,l}\rho^{m+2l} + o(\rho^k)\qquad \text{as }
\rho\to 0. 
\]
This shows that $u$ is $k$-times differentiable at $P_+$ if and only if 
$u_m$ satisfies 
the conditions \ref{it:oddeven} and \ref{it:firstderivatives} at $\theta=0$. 
Writing derivatives of $u$ of order $\leq k$ in terms of derivatives of $u_m$, we see these derivatives are continuous if $u_m\in C^k([0,\pi])$. 
The argument for the pole $P_-$ is analogous, with replacing the map $M$ by $-M$. 
\end{proof}

With the help of the following lemma, it can be shown that classical solutions 
to \eqref{eq:separated} are also weak solutions: 

\begin{lemma}\label{lem:symmetry}
For all $m\in\Z$ the operator $\Lapm^2$ is symmetric on 
the set 
\[
\left\{u\in C^4([0,\pi]):\Gamma_m u = 0, \Delta_m^2u \in L^2([0,\pi],r^2\sin)\right\}\subset L^2([0,\pi],r^2\sin).
\] 
\end{lemma}

\begin{proof}
Recalling the separated Laplacian in \eqref{eq:Lap_decomp} 
 and the weighted $L^2_{\diamond}$-inner product  allow us to perform the integration by part as
\begin{align*}
\lan \Lapm^2 u, v\ran&= -\lan \Gradm \Lapm u,\Gradm v \ran + \left[\sin\theta\,\dt (\Lapm u)\bar{v} \right]_{\theta\to 0}^{\theta\to \pi}\\
&=\lan\Lapm u, \Lapm v\ran + \left[\sin\theta\,\dt (\Lapm u)\bar{v} - \sin\theta\, \Lapm u\,\dt \bar{v}\right]_{\theta\to 0}^{\theta\to \pi}\\
&= \lan\Lapm u, \Lapm v\ran 
+ \frac{1}{r^2}\sum\nolimits_{j=1}^5 [a_j(\theta)]_{\theta\to 0}^{\theta\to \pi}
+\frac{m^2}{r^2} \sum\nolimits_{j=1}^3 [b_j(\theta)]_{\theta\to 0}^{\theta\to \pi}
\end{align*}
with
\begin{align*}
&a_1 = -\cos\theta\dt u \dt \bar{v},
&&a_2 = -\sin\theta\dt^2u \dt \bar{v},
&&a_3 = - \frac{1}{\sin\theta}\dt u \bar{v} ,
&&a_4 = \cos\theta\dt^2u \bar{v},\\
&a_5 = \sin\theta \dt^3u\bar{v},
&&b_1 = \frac{-1}{\sin\theta}\dt u \bar{v},
&&b_2 = \frac{2\cos\theta}{\sin^2\theta}u\bar{v},
&&b_3= \frac{1}{\sin\theta}u \dt\bar{v}
\end{align*}
for $u,v\in\Hd{s}$. Denoting the poles by $p\in\{0,\pi\}$ and considering $v$ having the same boundary conditions as $u$, we can see as follows that the boundary terms 
vanish for all $m\in\mathbb{Z}$:
\begin{description}
\item[$|m|\geq2$:] Here $u(p)=u'(p)=0=v(p)=v'(p)$ such that the limits 
for all the $a_j$-terms vanish. Moreover, we observe from l'Hospital's rule 
that  $\lim_{\theta\to p} b_2(\theta)= 2u'(p)\bar{v}'(p)=0$ and  
$\lim_{\theta\to p}b_j(\theta)=0$ for $j\in\{1,3\}$ such that all $b_j$-limits vanish. 
\item[$|m|=1$:] Here $u(p)=u''(p)=0=v(p)=v''(p)$, and hence 
$\lim_{\theta\to p}a_j(p)=0$ for $j\in \{2,4,5\}$.  Moreover, 
$\lim_{\theta\to p}(a_1+a_3)(\theta)= -2\cos(p)u'(p)\bar{v}'(p) = -
\lim_{\theta\to p}b_2(\theta)$, and 
$\lim_{\theta\to p} (b_1+b_3)(\theta)=0$ by l'Hospital.
\item[$m=0$:] Here $u'(p)=u'''(p)=0=v'(p)=v'''(p)$, and we only need to discuss
the $a_j$-terms, which vanish in the limit $\theta\to p$ for $j\in\{1,2,5\}$. 
Moreover, by l'Hospital 
$\lim_{\theta\to p} a_3(\theta)= -\cos(p)u''(p)\bar{v}(p)=- \lim_{\theta\to p}a_4(p)$.
\qedhere
\end{description}
\end{proof}

\begin{remark}[Restricted domain -- Cauchy boundary]\label{rem:symmetry-restricted}
It is clear that on the restricted domain $I'$ as in Section 
\ref{sec:wellposed-restricted} with boundary points $p\in \partial I'=\{\epsilon, \pi-\epsilon\}$, the symmetry property remains valid with Cauchy boundary conditions $u(p)=u'(p)=0$.
\end{remark}

\hl{
\section{On the assumptions of the uniqueness result}\label{appendix:baire}
In this appendix, we show 
that the conditions \eqref{nullset} 
of the uniqueness result in Corollary \ref{coro:uniqueness} are 
satisfied generically. That is, for a given $\Omega$, the set of 
right-hand sides $f$ for which \eqref{nullset} is violated is meagre
in $L^2_{\diamond}(\Mc;m)$, i.e., of first Baire category. Recall that by definition, this means that the set is a countable union of nowhere dense sets. By Baire's theorem, its complement is thus dense.

\begin{lemma}\label{lem:A_meagre}
Let $\emptyset\neq I'\subset (0,\pi)$ be open. Then 
$A:=\{\Phi:\lambda_{I'}(\{\Re\Phi=0\})>0\}$
is meagre in $\Hdm{2}$. Here, $\lambda_{I'}$ denotes the Lebesgue measure on $I'$, or, equivalently, the canonically weighted Lebesgue measure, while $\{\Re\Phi=0\}$ is a short-hand for $\{\theta\in I':\Re\Phi(\theta)=0\}$.
\end{lemma}

\begin{proof}
Given $n\in\mathbb{N}$, define $A_n:=\{\Phi:\lambda_{I'}(\{\Re\Phi=0\})>1/n\}$. We claim that $A_n$ is closed. For consider $\Phi\in A_n$ and a sequence $(\Phi_k)$ in $A_n$ such 
that $\lim_{k\to\infty}\|\Phi_k-\Phi\|_{H^2}=0$.  
Then by the definition of $A_n$, the sets $E_k:=\{\theta\in I':\Re\Phi_k(\theta)=0\}$ satisfy 
$\lambda_{I'}(E_k)\geq 1/n$. Set $E:=\limsup_{k\to\infty} E_k:=\bigcap_{k\geq 1}F_k$, with $F_k:=\bigcup_{l\leq k} E_l$. 
Since $\lambda_{I'}$ is finite and $F_1\supset F_2\supset \dots$, we have 
$\lambda_{I'}(E) = \lim_{k\to\infty} \lambda_{I'}(F_k) \geq \limsup_{k\to\infty} \lambda_{I'}(E_k)\geq 1/n$. If $\theta\in E$, then 
$\theta\in E_k$ for all but finitely many $k$, and since $\Phi_k$ converges pointwise to $\Phi$, 
we can conclude that $\Re\Phi(\theta)=0$. Hence $\Phi\in A_n$, so $A_n$ is closed. 

Next we show that $A_n$ does not contain any non-empty open set: If  
some ball around a function $\Psi\in A_n$ was a subset of $A_n$, then $A_n$ would contain 
an infinite number of the functions $\Psi-\frac{1}{k}$. This implies that infinitely many 
of the disjoint sets $\{\Re\Psi=1/k\}$ have measure $\geq 1/n$, contradicting the finiteness of $\lambda_{I'}$. 

As $A=\bigcup_{n=1}^{\infty}A_n$, this completes the proof.
\end{proof}

The following lemma shows meagreness of the set of sources for which 
the first condition in \eqref{nullset} is violated.
\begin{lemma}
The set $\{f\in L^2_{\diamond}(\Mc;m):\lambda_{I'}(\{\Delta_m\mathbf{B}_m^{-1}f=0\}>0\}$ is meagre in $L^2_{\diamond}(\Mc;m)$. 
\end{lemma}

\begin{proof}
The set $\{\Phi:\lambda_{I'}(\{\Phi=0\})>0\}$ is meagre in $\Hdm{2}$ 
as a subset of the set $A$ in Lemma \ref{lem:A_meagre}. 
Hence the statement follows from the fact that $\Delta_m\mathbf{B}_m^{-1}:L^2_{\diamond}(\Mc;m)\to \Hdm{2}$ is a homeomorphism.
\end{proof}

Finally, we show meagreness of the set of sources violating the 
second condition in \eqref{nullset}. 

\begin{lemma}
Assume that $\inf_{I'}|\Psi_0|>0$ and 
$\inf_{I'}|\Delta\Psi_0|>0$ for some interval 
$I'\subsetneq (0,\pi)$ and 
$\Psi_0:={\mathbf B}_m^{-1}f_0$ . Then there exists 
a ball $B_{f_0}$ in $L^2_{\diamond}(\Mc;m)$ centered at $f_0$ such that the following set is meagre in  $L^2_{\diamond}(\Mc;m)$:
\begin{align}\label{eq:phase_set}
\{f\in B_{f_0}:\lambda_{I'}(\{\operatorname{Arg}(\mathbf{B}_m^{-1}f)-\operatorname{Arg}(\Delta_m\mathbf{B}_m^{-1}f)\in\{-\pi,0,\pi\}\})>0\}.
\end{align}
\end{lemma}

\begin{proof}
Since $\Hdm{4}$ is continuously embedded in $L^{\infty}(\Mc;m)$, there 
exists a ball $B_{f_0}$ such that $\inf_{f\in B_{f_0}}\inf_{I'} |\mathbf{B}_m^{-1}f|>0$. Note that 
\[
\operatorname{Arg}\Psi - \operatorname{Arg}\Delta_m\Psi \in \{-\pi,0,\pi\}\notin\{-\pi,0,\pi\} 
\quad \Leftrightarrow\quad 
\frac{\Delta_m\Psi}{\Psi}\notin \mathbb{R}.
\]
For simplicity, first consider the case $I'=(0,\pi)$, noting that in view of the boundary conditions, this case can only occur for $m=0$. Then we define the mapping $G:D(G)\subset \Hdm{4}\to\Hdm{2}$ by $G(\Psi):=\Delta_m\Psi/\Psi$ on $D(G):=\{\Psi:\inf|\Psi|>0\}$. We claim
that 
\[
G'[\Psi_0]:\Hdm{4}\to \Hdm{2},\qquad 
G'[\Psi_0]\delta\Psi = \frac{1}{\Psi_0}(\Delta_m -G(\Psi_0)) \delta\Psi
\]
is boundedly invertible if $G(\Psi_0)$ is not 
an eigenvector of $\Delta_m$. In fact, since $\Hdm{2}$ is 
an algebra homomorphism and  $\Psi_0$ and $1/\Psi_0$ are bounded, 
the operator $\Phi \mapsto \Phi/\Psi_0$ is bounded and boundedly 
invertible in $\Hdm{2}$. Under our assumption, 
$\Delta_m-G(\Psi_0):\Hdm{4}\to \Hdm{2}$ is also bounded and 
boundedly invertible by Riesz theory. 
Hence, by the inverse function theorem, $G$ has a continuous 
inverse in a neighborhood of $\Psi_0$. Therefore, after possibly shrinking 
the ball $B_{f_0}$, $G\circ \mathbf{B}_m^{-1}$ restricted 
to $B_{f_0}$ is a homeomorphism. Since 
$A':=A\cap (G\circ \mathbf{B}_m^{-1})(B_{f_0})$ is meagre by 
Lemma \ref{lem:A_meagre}, so is $\mathbf{B}_m G^{-1}(A')$, which 
coincides with the set in \eqref{eq:phase_set}.

In the general case, we can choose Nemytskii operators 
$N_j(\Psi)(\theta):=n_j(\Psi(\theta),\theta)$, $j=1,2$ with 
smooth functions $n_j:D(n_j)\subset \mathbb{C}\times (0,\pi)\to \mathbb{C}$ 
satisfying $n_j(\Psi(\theta),\theta)=\Psi(\theta)$ for $\theta\in I'$ 
and $\Psi\in \mathbf{B}_m^{-1}(B_{f_0})$ and define 
$G(\Psi):=\frac{N_1(\Delta_m\Psi)}{N_2(\Psi)}$. Then 
$(N_j'[\Psi_0]\delta\Psi)(\theta) = \frac{\partial n_j}{\partial z}(\Psi_0(\theta),\theta) \,\delta \Psi(\theta)$ and 
\[
G'[\Psi_0]\delta\Psi = \frac{N_1'[\Psi_0]}{N_2(\Psi_0)}
\left(\Delta_m-G(\Psi_0)\frac{N_2'[\Psi_0]}{N_1'[\Psi_0]}\right)
\delta\Psi
\]
where we identify the multiplication operators $N_j'[\Psi_0]$ 
with their multiplier functions. Due to the assumptions 
on $\Psi_0$ and $\Delta\Psi_0$, after possibly shrinking 
$B_{f_0}$, we can choose $n_j$ such that 
$N_j(\Psi_0)$ and $N_j'[\Psi_0]$ are bounded and boundedly invertible 
in $\Hdm{2}$ and that $G(\Psi_0)\frac{N_2'[\Psi_0]}{N_1'[\Psi_0]}$ 
is not an eigenfunction of $\Delta_m$. We can then repeat the argument for the first case. 
\end{proof}


}

\printbibliography

@Book{adams:75,
  Title                    = {Sobolev Spaces},
  Author                   = {Robert A Adams},
  Publisher                = {Academic Press},
  Year                     = {1975},
  Address                  = {New York}
}

@article{Loptien2018,
  author    = {B. L\"{o}ptien and L. Gizon and A. C. Birch and J. Schou and B. Proxauf and T. L. Duvall Jr. and R. S. Bogart and U. R. Christensen},
  title     = {Global-scale equatorial Rossby waves as an essential component of solar internal dynamics},
  journal   = {Nature Astronomy},
  year      = {2018},
  volume    = {2},
  pages     = {568--574},
  doi       = {10.1038/s41550-018-0460-x},
  url       = {https://arxiv.org/abs/1805.07244}
}

@article{HNS95,
  author = {M. Hanke and A. Neubauer and O. Scherzer},
  title = {A convergence analysis of the {L}andweber iteration for nonlinear ill-posed problems},
  journal = {Numer.\ Math.}, 
  volume = {72},
  year = {1995},
  pages = {21-37}}

@article{Kindermann17,
title = {Convergence of the gradient method for ill-posed problems},
journal = {Inverse Problems \& Imaging},
volume = {11},
number = {4},
pages = {703-720},
year = {2017},
author = {Stefan Kindermann},
}

@book{KalNeuSch08,
  author = {Kaltenbacher, B. and Neubauer, A. and Scherzer, O.},
  title = {{ Iterative Regularization Methods for Nonlinear Problems}},
  publisher = {de Gruyter},
  address = {Berlin, New York},
  year = {2008},
  note = {Radon Series on Computational and Applied Mathematics}
}

@article{Kindermann21,
title = {On the tangential cone condition for electrical impedance tomography},
journal = {Electronic Transaction on Numerical Analysis},
volume = {57},
pages = {17--34},
year = {2022},
author = {Stefan Kindermann},
}

@book{Evans,
    author    = {L. C. Evans},
    title     = {Partial Differential Equations},
    year      = {1998},
    publisher = {Graduate Studies in Mathematics 19. AMS, Providence, RI},
    address   = {}
}

@book{TCC21,
  author={B. Kaltenbacher and T. T. N. Nguyen and O. Scherzer},
  title={The tangential cone condition for some coefficient identification model problems in parabolic {PDEs}},
   journal={Time-dependent Problems in Imaging and Parameter Identification, B. Kaltenbacher, T. Schuster, A. Wald (Eds.)},
  volume={},
  number={},
  pages={121–163},
  year={2021},
  publisher = {Springer},
  doi = "10.1007/10.1007/978-3-030-57784-1_5",
}

@article{HoffmanWaldNguyen:2021,
  author = {H. Hoffmann and A. Wald and T. T. N. Nguyen},
  title = {Parameter identification for elliptic boundary value problems: an abstract framework and application},
  journal = {Inverse Problems},
  pages = {44 pp},
  year = {2022},
  volume = {38},
  number ={7},
  doi = {10.1088/1361-6420/ac6d02},
}

@article{AarsetHollerNguyen23,
  title={{Learning-informed parameter identification in nonlinear time-dependent PDEs}},
  author={Aarset, Christian and Holler, Martin and Nguyen, Tram Thi Ngoc},
  
  journal={Applied Mathematics and Optimization},
  year={2023},
  volume = {88},
  pages = {53 pp},
  doi={10.1007/s00245-023-10044-y},
}

@article{HubmerScherzer:TCC18,
  author = {Simon Hubmer and Ekaterina Sherina and Andreas Neubauer and Otmar Scherzer},
  title = {{Lamé Parameter Estimation from Static Displacement Field Measurements in the Framework of Nonlinear Inverse Problems}},
  journal = {SIAM Journal on Imaging Sciences},
  year = {2018},
  volume = {11},
  number ={2},
  doi = {10.1137/17M1154461}
}

@article{Nakamura:MRE21,
  author = {Yu Jiang and Gen Nakamura and Kenji Shirota},
  title = {{Levenberg–Marquardt method for solving inverse problem of MRE based on the modified stationary Stokes system}},
  journal = {Inverse Problems},
  year = {2021},
  volume = {37},
  number ={12},
  eid ={125013}
}

@article{ScherzerHofmannNashed,
  author = {Otmar Scherzer and Bernd Hofmann and Zuhair Nashed},
  title = {{Gauss-Newton’s methods for solving linear inverse problems with neural network coders}},
  journal = {Sampl. Theory Signal Process. Data Anal.},
  pages = {},
  year = {2023},
  volume = {21},
  number = {25},
  doi = {10.1007/s43670-023-00066-6},
}

@article{HubmerRamlau17,
  author = {Simon Hubmer and Ronny Ramlau},
  title = {{Convergence analysis of a two-point gradient method for nonlinear ill-posed problems}},
  journal = {Inverse Problems},
  pages = {095004},
  year = {2017},
  volume = {33},
  number = {9},
  doi = {10.1088/1361-6420/aa7ac7},
}

@article{Neubauer17,
  author = {Andreas Neubauer},
  title = {{On Nesterov acceleration for Landweber iteration of linear ill-posed problems}},
  journal = {Journal of Inverse and Ill-posed Problems},
  pages = {381--390},
  year = {2017},
  volume = {25},
  number = {3},
  doi = {10.1515/jiip-2016-0060},
}

@article{Watson81,
  author = {M. Watson},
  title = {Shear instability of differential rotation in stars},
  journal = {Geophysical \& Astrophysical Fluid Dynamics},
  pages = {285--298},
  year = {1981},
  volume = {16},
  number = {1},
  doi = {10.1080/03091928008243663},
}

@article{Damien22,
  author = {Damien Fournier and Laurent Gizon and Laura Hyest},
  title = {Viscous inertial modes on a differentially rotating sphere: Comparison with solar observations},
  journal = {Astronomy \& Astrophysics},
  pages = {16 pp},
  year = {2022},
  volume = {664},
  doi = {10.1051/0004-6361/202243473},
  eid ={A6},
}

@article{GizonInertial21,
  author = {Laurent Gizon and others},
  title = {Solar inertial modes: Observations, identification, and diagnostic promise},
  journal = {Astronomy \& Astrophysics},
  pages = {22 pp},
  year = {2021},
  volume = {652},
  doi = {10.1051/0004-6361/202141462},
  eid ={L6},
}

@article{GizonBirchReview,
  author = {Laurent Gizon and Aaron C. Birch},
  title = {Local Helioseismology},
  journal = {Living Reviews in Solar Physics},
  pages = {131 pp},
  year = {2005},
  volume = {2},
  doi = {10.12942/lrsp-2005-6},
  eid = {6},
}

@book{Taylor,
    author    = {Michael E. Taylor},
    title     = {{Partial Differential Equations I}},
    year      = {2011},
    publisher = {Springer New York, NY},
    doi   = {10.1007/978-1-4419-7055-8},
    ISBN = {978-1-4614-2726-1},
}

@book{Hebey,
    author    = {Emmanuel Hebey},
    title     = {{Sobolev Spaces on Riemannian Manifolds}},
    year      = {1996},
    publisher = {Springer Berlin, Heidelberg},
    doi   = {10.1007/BFb0092907},
    ISBN = {978-3-540-69993-4},
}

@article{HaltmeierLeitaoScherzer,
   author={M Haltmeier and A Leitao and O Scherzer},
   title={{Kaczmarz methods for regularizing
nonlinear ill-posed equations I: Convergence Analysis}},
   journal={Inverse Problems and Imaging},
   volume={1},
   year={2007},
   pages={289--298},
}

@article{Bjoern23,
   author={B. M\"uller and T. Hohage and D. Fournier and
L. Gizon},
   title={{Quantitative passive imaging by iterative holography: The
example of helioseismic holography}},
   journal={},
   volume={40(4)},
   year={2024},
   pages={},
   doi   = {10.1088/1361-6420/ad2b9a},
}

@ARTICLE{Basu2016,
       author = {{Basu}, Sarbani},
        title = "{Global seismology of the Sun}",
      journal = {Living Reviews in Solar Physics},
         year = 2016,
        month = aug,
       volume = {13},
       number = {1},
          eid = {2},
        pages = {2},
          doi = {10.1007/s41116-016-0003-4},
archivePrefix = {arXiv},
       eprint = {1606.07071},
 primaryClass = {astro-ph.SR},
       adsurl = {https://ui.adsabs.harvard.edu/abs/2016LRSP...13....2B}
}

@ARTICLE{Gizon2020,
       author = {{Gizon}, Laurent and {Cameron}, Robert H. and {Pourabdian}, Majid and {Liang}, Zhi-Chao and {Fournier}, Damien and {Birch}, Aaron C. and {Hanson}, Chris S.},
        title = "{Meridional flow in the Sun{\textquoteright}s convection zone is a single cell in each hemisphere}",
      journal = {Science},
         year = 2020,
        month = jun,
       volume = {368},
       number = {6498},
        pages = {1469-1472},
          doi = {10.1126/science.aaz7119},
       adsurl = {https://ui.adsabs.harvard.edu/abs/2020Sci...368.1469G}
}

@INPROCEEDINGS{Gizon2024,
author = {{Gizon}, Laurent and {Bekki}, Yuto and {Birch}, Aaron C. and {Cameron}, Robert H. and {Fournier}, Damien and {Philidet}, Jordan and {Lekshmi}, B. and {Liang}, Zhi-Chao},
title = "{Solar Inertial Modes}",
booktitle = "{Dynamics of Solar and Stellar Convection Zones and Atmospheres
Proceedings IAU Symposium No. 365, A. V. Getling \& L. L. Kitchatinov, eds}",
year = 2024
}

@ARTICLE{Cowling1941,
       author = {{Cowling}, T.~G.},
        title = "{The non-radial oscillations of polytropic stars}",
      journal = {Monthly Notices of the Royal Astronomical Society},
         year = 1941,
        month = jan,
       volume = {101},
        pages = {367},
          doi = {10.1093/mnras/101.8.367},
       adsurl = {https://ui.adsabs.harvard.edu/abs/1941MNRAS.101..367C}
}

@article{Nguyen24,
doi = {10.1088/1361-6420/ad2905},
url = {https://dx.doi.org/10.1088/1361-6420/ad2905},
year = {2024},
publisher = {IOP Publishing},
volume = {40},
number = {4},
pages = {045020},
author = {Tram Thi Ngoc Nguyen},
title = {{Bi-level iterative regularization for inverse problems in nonlinear PDEs}},
journal = {Inverse Problems},
}

@article{Thompson03,
doi = {10.1146/annurev.astro.41.011802.094848 },
year = {2003},
volume = {41},
pages = {599--643},
author = { Thompson, Michael J. and Christensen-Dalsgaard, Jørgen and Miesch, Mark S. and Toomre, Juri },
title = {{The Internal Rota-
tion of the Sun}},
journal = {Annu. Rev. Astron. Astrophys.},
}

@article{Kaltenbacher23-rangeinv,
doi = {10.1093/imanum/drad044},
year = {2023},
volume = {drad044},
pages = {},
author = {Barbara Kaltenbacher},
title = {{Convergence guarantees for coefficient reconstruction in PDEs from boundary measurements by variational and Newton-type methods via range invariance}},
journal = {IMA Journal of Numerical Analysis},
}

@article{PhilidetGizon23,
doi = {10.1051/0004-6361/202245666 },
year = {2023},
volume = {673},
number = {},
pages = {19 pp},
author = {J. Philidet and L. Gizon},
title = {{Interaction of solar inertial modes with turbulent convection. 
A 2D model for the excitation of linearly stable modes}},
journal = {Astronomy \& Astrophysics},
}

@article{Bekki2022,
       author = {{Bekki}, Yuto and {Cameron}, Robert H. and {Gizon}, Laurent},
        title = "{Theory of solar oscillations in the inertial frequency range: Amplitudes of equatorial modes from a nonlinear rotating convection simulation}",
      journal = {Astronomy \& Astrophysics},
         year = 2022,
        month = oct,
       volume = {666},
          eid = {A135},
        pages = {A135},
          doi = {10.1051/0004-6361/202244150},
archivePrefix = {arXiv},
       eprint = {2208.11081},
 primaryClass = {astro-ph.SR},
       adsurl = {https://ui.adsabs.harvard.edu/abs/2022A&A...666A.135B}
}

@article{EllerRolandRieder24,
author = {Eller, Matthias and Griesmaier, Roland and Rieder, Andreas},
title = {{Tangential Cone Condition for the Full Waveform Forward Operator in the Viscoelastic Regime: The Nonlocal Case}},
journal = {SIAM Journal on Applied Mathematics},
volume = {84},
number = {2},
pages = {412-432},
year = {2024},
doi = {10.1137/23M1551845},
URL = {https://doi.org/10.1137/23M1551845},
eprint = {https://doi.org/10.1137/23M1551845}
}

@article{VersyptBraatz14,
author = {Versypt, AN and Braatz, RD},
title = {{Analysis of Finite Difference Discretization Schemes for Diffusion in Spheres with Variable Diffusivity}},
journal = {Comput Chem Eng.},
volume = {71},
pages = {241--252},
year = {2014},
doi = {10.1016/j.compchemeng.2014.05.022}
}

@article{modelS,
        author = {Christensen-Dalsgaard, J. and D{\"a}ppen, W. and Ajukov, S. V.
                  and Anderson, E. R. and Antia, H. M. and Basu, S. and Baturin, V. A. and Berthomieu, G.
                  and Chaboyer, B. and Chitre, S. M. and Cox, A. N. and Demarque, P.
                  and Donatowicz, J. and Dziembowski, W. A. and Gabriel, M. and Gough, D. O.
                  and Guenther, D. B. and Guzik, J. A. and Harvey, J. W. and Hill, F. and Houdek, G.
                  and Iglesias, C. A. and Kosovichev, A. G.
                  and Leibacher, J. W. and Morel, P. and Proffitt, C. R. and Provost, J. and Reiter, J. and Rhodes, E. J.
                  and Rogers, F. J. and Roxburgh, I. W. and Thompson, M. J. and Ulrich, R. K.},
        title = {The Current State of Solar Modeling},
        volume = {272},
        number = {5266},
        pages = {1286--1292},
        year = {1996},
        note = {10.1126/science.272.5266.1286},
        publisher = {American Association for the Advancement of Science},
        issn = {0036-8075},
        journal = {Science}
}

@Book{nedelec:01,
  author    = {J C N\'ed\'elec},
  publisher = {Springer},
  title     = {Acoustic and Electromagnic Equations: Integral Respresentations for Harmonic Problems},
  year      = {2001},
  address   = {New York},
}

@Book{Renardy1992,
  author    = {Michael Renardy and Robert C Rogers},
  publisher = {Springer},
  title     = {An Introduction to Partial Differential Equations},
  year      = {1992},
}

@Book{zeidler1986,
  author    = {Zeidler, Eberhard},
  publisher = {Springer-Verlag},
  title     = {Nonlinear Functional Analysis and Its Applications {I}: Fixed-point theorems},
  year      = {1986},
  address   = {New York},
  doi       = {10.1007/978-1-4612-4838-5},
  pages     = {xxi+897},
  url       = {http://dx.doi.org/10.1007/978-1-4612-4838-5},
}

@book{Triebel,
    author    = {Triebel, Hans},
    title     = {Interpolation theory, function spaces, differential operators},
    year      = {1978},
    publisher = {Amsterdam ; New York : North-Holland Pub. Co.},
    address   = {}
}

@article{BadrBernicotRuss,
       author = {Badr, N. and Bernicot and F. and Russ, E.},
        title = "{ Algebra properties for Sobolev spaces - applications to semilinear PDEs on manifolds}",
      journal = {JAMA},
         year = 2012,
       volume = {118},
        pages = {509--544},
          doi = {10.1007/s11854-012-0043-1},
}

@article{nguyen25,
doi = {10.1088/1361-6420/addf73},
year = {2025},
publisher = {IOP Publishing},
volume = {41},
number = {6},
pages = {},
author = {T. T. N. Nguyen},
title = {{Sequential bi-level regularized inversion with application to hidden reaction law discovery}},
journal = {Inverse Problems}
}

@ARTICLE{Liang2025,
       author = {{Liang}, Zhi-Chao and {Gizon}, Laurent},
        title = "{Doppler velocity of m = 1 high-latitude inertial mode over the last five sunspot cycles}",
      journal = {Astronomy \& Astrophysics},
         year = 2025,
        month = mar,
       volume = {695},
          eid = {A67},
        pages = {A67},
          doi = {10.1051/0004-6361/202452133},
archivePrefix = {arXiv},
       eprint = {2409.06896},
 primaryClass = {astro-ph.SR},
       adsurl = {https://ui.adsabs.harvard.edu/abs/2025A&A...695A..67L}
}

@ARTICLE{Sup2025,
       author = {{Mukhopadhyay}, Suprabha and {Bekki}, Yuto and {Zhu}, Xiaojue and {Gizon}, Laurent},
        title = "{Assessing the validity of the anelastic and Boussinesq approximations to model solar inertial modes}",
      journal = {Astronomy \& Astrophysics},
         year = 2025,
        month = apr,
       volume = {696},
          eid = {A160},
        pages = {A160},
          doi = {10.1051/0004-6361/202453634},
archivePrefix = {arXiv},
       eprint = {2501.16797},
 primaryClass = {astro-ph.SR},
       adsurl = {https://ui.adsabs.harvard.edu/abs/2025A&A...696A.160M}
}

@ARTICLE{Bekki2024,
       author = {{Bekki}, Yuto and {Cameron}, Robert H. and {Gizon}, Laurent},
        title = "{The Sun's differential rotation is controlled by high-latitude baroclinically unstable inertial modes}",
      journal = {Science Advances},
         year = 2024,
        month = mar,
       volume = {10},
       number = {13},
          eid = {eadk5643},
        pages = {eadk5643},
          doi = {10.1126/sciadv.adk5643},
archivePrefix = {arXiv},
       eprint = {2403.18986},
 primaryClass = {astro-ph.SR},
       adsurl = {https://ui.adsabs.harvard.edu/abs/2024SciA...10K5643B}
}

@article{nguyen-passive,
      title={The extended adjoint state and nonlinearity in correlation-based passive imaging}, 
      author={Tram Thi Ngoc Nguyen},
 journal = {SIAM Journal on Applied Mathematics},
      year={(to appear)},
      eprint={2504.16797},
      archivePrefix={arXiv},
      primaryClass={math.NA},
      url={https://arxiv.org/abs/2504.16797}, 
}
\end{document}